\documentclass[times]{oupau_arxiv}
\usepackage{mathtools}
\usepackage{xcolor}
\usepackage{amssymb}
\usepackage{enumitem}
\usepackage{natbib}
\usepackage[utf8]{inputenc}
\usepackage[cyr]{aeguill}
\usepackage{tikz-cd}
\usepackage[pdfdisplaydoctitle=true,
            colorlinks=true,
            urlcolor=blue,
            citecolor=blue,
            linkcolor=blue,
            pdfstartview=FitH,
            pdfpagemode= UseNone,
            bookmarksnumbered=true]{hyperref}
\usepackage{todonotes}
\usepackage[all]{xy}
\usepackage{stmaryrd}
\usepackage{comment}
\mathtoolsset{showonlyrefs}
\usepackage{setspace}
\usepackage{lmodern}

\newtheorem{remark}[theorem]{Remark}

\newtheorem*{theorem*}{Theorem}

\newcommand{\R}{\mathbb R}

\newcommand{\Dens}{\operatorname{Dens}_+(M)}

\newcommand{\define}[1]{\textbf{#1}}

\newcommand{\WFRAF}{{\textrm{WFR}_{\delta}^{\eich,\ef}}}
\newcommand{\WFRAFho}{{\textrm{WFR}_{\delta}^{\eich,0}}}
\newcommand{\WFRGM}{{\textrm{WFR}_{\delta}^{1,\ef}}}
\newcommand{\WFRconst}{{\textrm{WFR}_{\delta}^{\eich,\Ce}}}

\newcommand{\Meas}{\mathcal{M}}
\newcommand{\Mplus}{\Meas^+}
\newcommand{\Mplushf}{\Meas^+_{\eich,C}(\Omega)}
\newcommand{\MplushC}{\Meas^+_{\eich,C}(\Omega)}

\newcommand{\rhosq}{(\rho_0,\rho_1)^2}

\newcommand{\eich}{H}
\newcommand{\Htx}{\eich(t,x)}

\newcommand{\ef}{F}
\newcommand{\Ft}{\ef(t)}
\newcommand{\Ce}{C}

\newcommand{\CE}{\mathcal{CE}(\rho_0,\rho_1)}
\newcommand{\CEHC}{\mathcal{CE}_{\eich,\Ce}(\rho_0,\rho_1)}
\newcommand{\CEHF}{\mathcal{CE}_{\eich,\ef}(\rho_0,\rho_1)}

\newcommand{\AFF}{\mathcal{A}}
\newcommand{\AHC}{\AFF(\eich,\Ce)}
\newcommand{\AHF}{\AFF(\eich,\ef)}

\title{Path constrained unbalanced optimal transport}

\volumeyear{2009}
\paperID{rnn999}

\author{Martin Bauer\affil{1}, Nicolas Charon\affil{2}, Tom Needham\affil{1}, and Mao Nishino\affil{1}}
\abbrevauthor{M. Bauer, N. Charon, T. Needham, and M. Nishino}
\headabbrevauthor{M. Bauer, N. Charon, T. Needham, and M. Nishino}

\address{%
\affilnum{1}Florida State University,
Tallahassee, Florida 32306
\affilnum{2}University of Houston, Houston, Texas 77204}

\correspdetails{ncharon@central.uh.edu}

\received{1 Month 20XX}
\revised{11 Month 20XX}
\accepted{21 Month 20XX}

\begin{abstract}
Dynamical formulations of optimal transport (OT) frame the task of comparing distributions as a variational problem which searches for a path between distributions minimizing a kinetic energy functional. In applications, it is frequently natural to require paths of distributions to satisfy additional conditions. Inspired by this, we introduce a model for dynamical OT which incorporates constraints on the space of admissible paths into the framework of unbalanced OT, where the source and target measures are allowed to have a different total mass. Our main results establish, for several general families of constraints, the existence of solutions to the variational problem which defines this path constrained unbalanced optimal transport framework. These results are primarily concerned with distributions defined on an Euclidean space, but we  extend them to distributions defined over parallelizable Riemannian manifolds as well. We also consider metric properties of our framework, showing that, for certain types of constraints, our model defines a metric on the relevant space of distributions. This metric is shown to arise as a geodesic distance of a Riemannian metric, obtained through an analogue of Otto’s submersion in the classical OT setting.

\end{abstract}

\begin{document}
\maketitle

\setcounter{tocdepth}{1}


\section{Introduction}
\paragraph{Background:}
The field of optimal transport (OT) studies various methods for comparing probability distributions, generally guided by the intuition that the comparison should be performed by measuring the optimal energy required to transport mass from one distribution to the other. It is a well-established area of study in both pure and applied mathematics~\cite{villani2009optimal}, and has more recently seen a surge of interest due to its applications in, for example,  geometry processing~\cite{solomon2015convolutional} and machine learning~\cite{peyre2019computational}. The classical Kantorovich formulation of OT compares probability distributions $\rho_0$ and $\rho_1$ by searching for a \emph{coupling} of the distributions—that is, a joint distribution with marginals $\rho_0$ and $\rho_1$—which is optimal with respect to some cost. There are natural choices of the cost function in certain situations (for example, when the distributions are Borel measures on the same metric space), and under these circumstances, the optimal cost leads to a metric on the space of distributions referred to as \emph{Wasserstein distance}. 
Motivated by applications in which the pure transport assumption can be an unrealistic and limiting factor, this framework has been recently extended to an unbalanced setting, which allows for the creation or destruction of mass via a local source term~\cite{chizat2018interpolating,liero2018optimal,chapel2020partial}. 
An immediate advantage of these \emph{unbalanced optimal transport problems}, which play a central role in the present paper, is their ability to facilitate the comparison of densities with uneven mass.
%


Another perspective on OT which is important for this paper is the dynamical one. When considering densities on $\R^n$, a famous result of Benamou and Brenier~\cite{benamou_computational_2000} states that the Wasserstein distance can be realized as the ``kinetic energy” of an energy-minimizing path of densities (i.e. a time-evolving density) from the source to the target, where paths are subject to a certain PDE encoding constraints on the evolution.

\paragraph{Central Goal:}
The goal of this paper is to understand basic analytical and geometrical properties of a certain variant of this formulation---the \define{constrained unbalanced optimal transport problem}---which is both \emph{unbalanced}, in the sense that it is able to compare distributions with different total mass, and \emph{constrained}, in the sense that additional requirements are imposed on the evolving path of densities. In particular, our main results establish the existence of minimizers for the associated optimization problem and give conditions under which it leads to a metric on a relevant space of measures.

\paragraph{Motivation:}
The motivation for these investigations comes from the general idea that practical OT problems can naturally come with additional constraints on the space of measures being considered. In certain cases, those constraints are preserved along the geodesics of the OT metrics. This is notably the case when one is interested in comparing Gaussian distributions over $\mathbb{R}^d$ based on the Wasserstein metric, for which the space of Gaussian distributions is totally geodesic. Yet this situation is rather exceptional, and there are many examples of fairly simple and natural constraints which are not preserved along optimal paths for balanced or unbalanced optimal transport. For instance, in the unbalanced OT model described above, the subspace of probability measures is not totally geodesic as evidenced by some of the examples studied in \cite{chizat2018interpolating}. In fact, it has been recently shown that, for a certain choice of energy functional, the unbalanced OMT model can be viewed as a conic extension over the space of probability densities~\cite{laschos2019geometric}. This interpretation allows the derivation of an explicit formula for the corresponding distance on the base manifold -- the space of probability densities -- and has been used to develop numerical algorithms allowing to use the unbalanced energy functional for applications in random sampling~\cite{jing2024machine}. We will describe this particular example of constraint unbalanced OMT problem in more detail in Section~\ref{ex:WFR_prob}.

{Another example is the situation of \textit{area measures} of convex domains in $\mathbb{R}^d$. In fact, this example initiated the present work, following the preliminary analysis presented in \cite{charon_length_2021} for the one-dimensional case.} These are measures defined on the unit sphere $S^{d-1}$ representing the distribution of {the} area of the domain's boundary along the different normal directions. As suggested in \cite{charon_length_2021}, building a Riemannian metric on the space of area measures induces a Riemannian metric over the set of all convex sets via the correspondence given by the Minkowski-Fenchel-Jessen theorem, which in turn may be used to evaluate distances and recover geodesics between such sets. Although natural at first sight, standard balanced or unbalanced optimal transport metrics on $S^{d-1}$ are not directly suited for that purpose since area measures must satisfy a specific ``closure'' constraint, namely that the first moment vanishes; a condition that is not preserved along optimal paths for the Wasserstein or WFR metrics, cf. Example~\ref{ex:area_measures}. More generally, one may be interested in the comparison of measures with some prescribed set of moments or covariance structures. 

Moreover, beyond the case of static/stationary constraints, another plausible scenario is to observe the time evolution of a specific quantity -- such as the total mass, cf. Example~\ref{ex:totalmass_movingbarrier}, of the measure or the expectation of some function -- and then attempt to find an optimal interpolating path enforcing this set of constraints. Finally, the present framework would also allow {us} to model the situation of a moving obstacle, i.e., unbalanced optimal transport with a (potentially) time evolving barrier within the domain,  cf. Example~\ref{ex:totalmass_movingbarrier}.
The common problem in all such examples is coming up with a constrained version of OT or UOT that can incorporate path constraints into the dynamic formulation. We should emphasize that several past works have considered related constrained variants of dynamic optimal transport models, in particular \cite{papadakis2014optimal,kerrache2022constrained,neklyudov2023computational}. 
Yet, the aforementioned references have primarily focused on deriving numerical methods to tackle specific cases. A systematic study of the theoretical properties of constrained optimal transport, especially in the unbalanced setting{,} remained missing and is the content of the present article. Of particular importance is the question of the existence of optimal paths, which, as we shall see, cannot be deduced straightforwardly from the corresponding results in standard optimal transport and is one of the central contributions of the present work. 

We emphasize that we work in the setting of \emph{dynamical} OT, where constraints are used to define admissible paths of measures which interpolate between the target measures. An alternative approach is to apply constraints in the static (Kantorovich) formulation of OT by imposing additional conditions on couplings between the target measures. For example, this is a useful perspective in mathematical finance, where it is natural to impose \emph{martingale constraints} on transport plans~\cite{beiglbock2013model,galichon2014stochastic,ekren2018constrained}. This type of framework is also used to compare Gaussian mixture models, in which case one obtains a natural metric by also restricting couplings to be themselves Gaussian mixtures~\cite{delon2020wasserstein,wilson2023wasserstein}. Although these formulations of constrained static OT are quite useful, they are unable to capture the more detailed dynamic constraints described in the previous paragraph. The constraint model we consider in this paper is distinct from and complementary to those static models. To the best of our knowledge, the only (dynamically) constrained model studied to date is the case where unbalanced OMT is restricted to the space of probability densities~\cite{laschos2019geometric}, as described above. The goal of the present article is to provide a general framework allowing different types of (application-oriented) constraints to be incorporated into the unbalanced OMT problem. Next{,} we shall briefly summarize the proposed mathematical model.

\paragraph{Mathematical Model:}
To describe our results in more detail, let us introduce some notation. Let $\Omega \subset \R^n$ be a compact set with $C^1$ boundary and let $u_0,u_1:\Omega \to \R$ be (sufficiently smooth) probability densities with respect to Lebesgue measure. The \emph{Benamou-Brenier formulation of OT} solves the optimization problem
\begin{equation}\label{eqn:benamou_brenier_intro}
\inf_{u,v} \int_0^1 \int_\Omega \frac{\|v(t,x)\|^2}{u(t,x)} \; dx \, dt \quad \mbox{subject to} \quad \partial_t u + \nabla \cdot v = 0, \; u(0,\cdot) = u_0, \; u(1,\cdot) = u_1, 
\end{equation}
where the infimum is taken over paths of densities $u:[0,1] \times \Omega \to \R$ and vector fields $v:[0,1] \times \Omega \to \R^n$. This dynamical formulation of OT can be extended to compare densities $u_0$ and $u_1$ which do not necessarily have the same mass. Following~\cite{chizat2018interpolating}, one modifies the Benamou-Brenier formulation of OT by introducing a \emph{source term} $w:[0,1] \times \Omega \to \R$; the \emph{unbalanced OT problem} is then written as 
\begin{equation}\label{eqn:unbalanced_OT_intro}
\inf_{u,v,w} \int_0^1 \int_\Omega \frac{\|v(t,x)\|^2 + \delta^2 w(t,x)^2}{u(t,x)} \; dx \, dt \quad \mbox{subject to} \quad \partial_t u + \nabla \cdot v = w, \; u(0,\cdot) = u_0, \; u(1,\cdot) = u_1, 
\end{equation}
where $\delta \geq 0$ is a hyperparameter controlling the penalty for the creation or destruction of mass. The PDE constraint is referred to as the \emph{continuity equation}.

 To address the constrained situations described above, we consider the unbalanced OT energy with additional constraints on the path of densities $u$. To be more specific, we consider  \emph{affine constraints} on $u$ of the form
\begin{equation}\label{eqn:affine_constraint_intro}
\int_\Omega \Htx u(t,x) dx = \Ft
\end{equation}
for some functions $\eich:[0,1] \times \Omega \to \R^d$ and $\ef:[0,1] \to \R^d$ (for some $d \geq 1$). We say that an affine constraint is \emph{time-independent} if $\ef$ is equal to a constant $\Ce \in \R^d$ and the value of $\Htx$ does not depend on $t$ (in which case we consider it as a function $\eich:\Omega \to \R$), and \emph{time-dependent} otherwise. The \define{unbalanced OT problem with affine constraint} associated to $\eich$ and $\ef$ is then given by the optimization problem \eqref{eqn:unbalanced_OT_intro} subject to the additional constraint \eqref{eqn:affine_constraint_intro}---in fact, the general problem we consider has a slightly more technical formulation which allows comparison of arbitrary positive Radon measures (see Definitions \ref{def:WFR_problem_constant_time} and \ref{def:wfr_problem_time_varying}). In particular, this model encompasses the motivating constrained UOT scenarios described above, as summarized in Figure~\ref{fig:example}.

\begin{figure}[t]
    \centering
    \includegraphics[width=\textwidth]{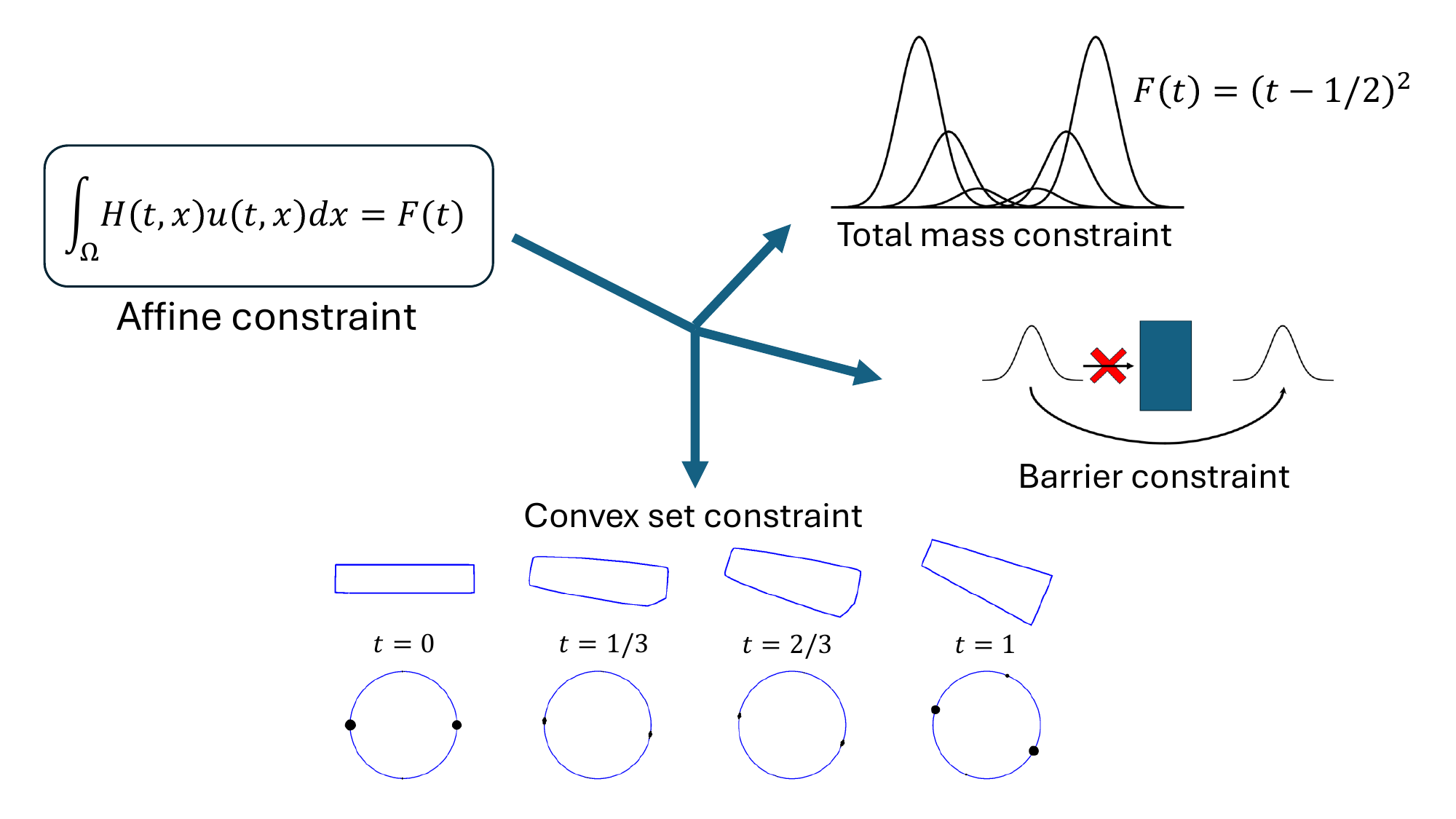}
    \caption{Examples of affine constraint unbalanced OMT models: The total mass constraint corresponds to  $H=1$, and $F(t)$ prescribing the amount of total mass at each time $t$. 
    The spherical Hellinger-Kantorovich distance, as studied in~\cite{laschos2019geometric,jing2024machine}, corresponds to the special case $F(t)=1$. 
   We also show a schematic of the barrier constraint, which disallows the existence of mass in a certain region of the domain $\Omega$. These two constraints are discussed in detail in Example \ref{ex:totalmass_movingbarrier}. The figure for the convex set constraint was adapted from \cite{charon_length_2021}, where the bottom row shows the transportation of Dirac masses on $S^1$ while the top row shows the convex curves associated {with} these measures. This involves a special case of affine constraint enforcing {the} closure of the curves, as discussed in Example \ref{ex:area_measures}.}
    \label{fig:example}
\end{figure}

\paragraph{Main Contributions:}
We now describe our main results. In the following we call $(u,v,w)$ a \emph{finite energy path} from $u_0$ to $u_1$, with respect to $\eich$ and $\ef$, if $u(0,\cdot) = u_0$, $u(1,\cdot) = u_1$, $u$ satisfies the affine constraint \eqref{eqn:affine_constraint_intro} and the continuity equation in~\eqref{eqn:unbalanced_OT_intro}, and the objective of \eqref{eqn:unbalanced_OT_intro} takes a finite value on $(u,v,w)$. We say that $(u,v,w)$ is a \emph{minimal energy path} with respect to $\eich$ and $\ef$ if, in addition, it realizes the infimum in~\eqref{eqn:unbalanced_OT_intro}. These concepts generalize to positive Radon measures, and the statements of our results are given in this context. 

Our first main result gives general conditions which guarantee the existence of a minimizer to the constrained unbalanced optimal transport problem:
\begin{theorem*}[Existence of minimizers; Theorem 
    \ref{thm: AffWFRDuality_time}]
    Let $\rho_0$ and $\rho_1$ be positive Radon measures on $\Omega$, let $\eich:[0,1]\times \Omega \to \mathbb{R}^d$ be a continuous function and let $\ef=(\ef_1,\ldots,\ef_d):[0,1]\to \mathbb{R}^d$ be a function with integrable component functions $\ef_i$. If a finite energy path from $\rho_0$ to $\rho_1$ with respect to $\eich$ and $\ef$ exists, then a minimal energy path exists.
\end{theorem*}

We give several classes of functions $\eich$ and $\ef$ which guarantee the existence of finite (hence minimal) energy paths in Theorems \ref{thm: Affh0 Finite Energy Path_time_constant}, \ref{thm:general_finite_energy_existence} and 
\ref{prop:Affh0_Finite_Energy_Path}. The statements of the results for time-dependent functions are somewhat technical, so we only state here the theorem in the case of time-independent constraints. For $\eich:\Omega \to \R^d$ and $\Ce \in \R^d$, let $\MplushC$ denote the space of positive Radon measures $\rho$ satisfying $\int_\Omega \eich \; d\rho = \Ce$.

\begin{theorem*}[Existence of finite energy paths; Theorem \ref{thm: Affh0 Finite Energy Path_time_constant}]
    Let $\eich: \Omega \to \mathbb{R}^d$ be a continuous function, let $C \in \R^d$ and let $\rho_0,\rho_1 \in \MplushC$. Assume, in addition, that one of the following holds:
    \begin{enumerate}
     \item $\Omega$ is convex, $d=1$ and $\eich\in C^1(\Omega,[c,\infty))$ for some $c>0$,
    \item $C=0$, or 
    \item there exists a density $\bar \rho$ with respect to which the Radon-Nikodym derivative of both $\rho_0$ and $\rho_1$ is a continuous and positive function, i.e., $\frac{\rho_i}{\bar \rho}\in C(\Omega,\mathbb R_{>0})$, {where $\mathbb R_{>0}$ denotes the set of strictly positive real numbers.}
    \end{enumerate}
     Then there is a finite energy path from $\rho_0$ to $\rho_1$ with respect to $\eich$ and the constant function $\ef = C$. 
\end{theorem*}

In the time-independent case, we further show that the constrained unbalanced OT problem leads to a metric and describe the corresponding geometric picture, including geodesic equations in the smooth positive density setting (Section \ref{sec:constrained_WFR_distance}) as well as an analogue of Otto's Riemannian submersion formalism (Section~\ref{ssec:Otto_picture}): 
\begin{theorem*}[Theorem \ref{thm:constrained_WFR_metric}] Let ${H}:\Omega \to \mathbb{R}^d$ and $\Ce\in \mathbb{R}^d$ be such that there is a finite energy path between any two measures in $\MplushC$. Then, the constrained unbalanced OT problem induces a metric on $\MplushC$, and this metric can be interpreted as the geodesic distance of a (formal) Riemannian metric. 
\end{theorem*}
Lastly, in Section~\ref{sec:manifolds}, we extend both the model and the existence results to the case, where the domain is instead a Riemannian manifold. For the proof of the existence of minimizers result, cf. Theorem~\ref{thm:AffWFRDuality_time_M}, we make the assumption that the manifold is parallelizable; the corresponding existence of finite energy path results translate without any additional assumption to the manifold situation. 


\paragraph{Future Directions:} This paper is primarily focused on the fundamental theoretical aspects of the constrained model. Our framework suggests a computational scheme based on the convexity of the formulation, and we intend to develop such a numerical implementation in future work. We also plan to extend our model to incorporate other variants of unbalanced transport. While we specifically focus on a constrained version of the Wasserstein-Fisher-Rao distance (see Definition~\ref{def:wfr_distance}), there are other dynamic unbalanced optimal transport distances which should be amenable to constrained formulations. These include the various distances described in~\cite{chizat2018interpolating} or the recently introduced \emph{simple unbalanced optimal transport} model of~\cite{khesin2023simple}. The latter model could be interesting in that it may allow for a more explicit description of the corresponding constrained geometric picture. Although we were primarily interested here in the situation of constraints on the measure path only, one could more generally introduce similar affine constraints on the vector field and source term instead. We leave to future investigations the study of how the existence results and proofs of this work can adapt to this alternative setting. Lastly, while affine equality constraints are being considered here, generalizing this framework to convex inequality constraints is another interesting path with the potential to greatly extend the scope of applications of this model. 

\paragraph{Acknowledgements:}
M. Bauer and M. Nishino were partially supported by NSF grant CISE-2426549. T.~Needham was partially supported by NSF grants DMS-2107808 and DMS-2324962. N. Charon was partially supported by NSF grants DMS-2402555 and DMS-2438562.

\section{Preliminaries}

Let us begin by setting notation that will be used globally throughout the paper; more specialized notation will be introduced as needed. We will always use $\Omega \subset \R^n$ to denote a fixed compact set with $C^1$ boundary. For certain results, we also require $\Omega$ to be convex, but this condition is not generally assumed. For a space $X$, we use $\Meas(X)$ to denote the set of Radon measures and $\Meas^+(X)$ to denote the set of positive Radon measures on $X$. We generally use standard notation for concepts from analysis and measure theory without additional explanation; e.g., $\mu_0 \ll \mu_1$ if a measure $\mu_0$ is absolutely continuous with respect to a measure $\mu_1$.

The rest of this subsection fills in details on the background topics of optimal transport that were described in the introduction.

\subsection{Dynamic optimal transport and Fisher-Rao distance.} The standard Kantorovich formulation of optimal transport compares probability measures $\rho_0$ and $\rho_1$ on $\Omega \subset \R^n$ by solving the optimization problem
\begin{equation}\label{eqn:kantorovich_ot}
\inf_{\gamma \in \Gamma(\rho_0,\rho_1)} \int_{\Omega \times \Omega} \|x- y\|^2 d\gamma(x,y),
\end{equation}
where $\Gamma(\rho_0,\rho_1)$ is the set of probability measures on $\Omega \times \Omega$ with marginals $\rho_0$ and $\rho_1$, respectively. More generally, one can consider other costs in the integrand, or even work with measures over general metric spaces, but the formulation provided here is sufficient to illustrate the idea. The square root of \eqref{eqn:kantorovich_ot} defines a metric on the space of measures, referred to as the \define{Wasserstein distance}. 

A classic result of optimal transport theory due to Benamou and Brenier~\cite{benamou_computational_2000} explains how the Kantorovich problem \eqref{eqn:kantorovich_ot} can be recast in a dynamical framework. Suppose that each measure $\rho_i$ has density $u_i:\Omega \to \R$ with respect to {the} Lebesgue measure. We consider a path of densities $u:[0,1] \times \Omega \to \R$, together with a time-dependent vector field $v:[0,1] \times \Omega \to \R^n$; we assume that all maps are sufficiently smooth in time and space. The Benamou-Brenier result~\cite[Proposition 1.1]{benamou_computational_2000} says that solving \eqref{eqn:kantorovich_ot} is equivalent to solving
\begin{equation}\label{eqn:dynamic_ot}
\inf_{u,v} \int_0^1 \int_\Omega \frac{\|v(t,x)\|^2}{u(t,x)} \; dx \, dt,
\end{equation}
subject to the \define{continuity equation}:
\begin{equation}\label{eqn:continuity_equation}
\partial_t u + \nabla \cdot v = 0, \qquad u(0,\cdot) = u_0, \; u(1,\cdot) = u_1.
\end{equation}

Another approach to metrizing the space of densities is given by the Fisher-Rao distance. This metric was originally introduced to compare elements of finite-dimensional submanifolds of probability densities in the context of statistics~\cite{radhakrishna1945information}, but has since been extended to allow comparison of more general densities (see, e.g.,~\cite{friedrich1991fisher,bauer2016uniqueness}). Given two positive densities $u_0,u_1:\Omega \to \R$, the \define{Fisher-Rao distance} between them is given by
\begin{equation}\label{eqn:fisher_rao_distance}
\inf_{u,w} \int_0^1 \int_\Omega \frac{w(t,x)^2}{u(t,x)} \; dx \, dt, \qquad \mbox{subject to} \quad \partial_t u = w, \quad u(0,\cdot) = u_0, \; u(1,\cdot) = u_1,
\end{equation}
where the infimum is over sufficiently smooth paths of densities $u$ and functions $w:[0,t] \times \Omega \to \R$. There is an apparent similarity of structure between the Fisher-Rao distance and the dynamical formulation of OT; the connection between these metrics is elucidated below. 

\subsection{Unbalanced dynamic optimal transport and the Wasserstein-Fisher-Rao distance.}
A dynamical formulation of \emph{unbalanced} optimal transport---i.e., where the measures being compared are allowed to have different total mass---was introduced in \cite{chizat2018interpolating}. The idea is to relax the continuity equation \eqref{eqn:continuity_equation} in the Benamou-Brenier formulation of optimal transport by including a \emph{source term} which allows the creation and destruction of mass, and to then incorporate the source into the objective function \eqref{eqn:dynamic_ot}. The modified continuity equation is
\begin{equation}\label{eqn:modified_continuity_eqn}
    \partial_t u + \nabla \cdot v = w,
\end{equation}
where $w:[0,1] \times \Omega \to \R$ is the source term, and the modification of the objective function is to be defined. This not only allows one to compare positive measures with different mass (a useful extension in applications), but also frequently results in more natural interpolations between densities---see the recent survey~\cite{sejourne2022unbalanced} for an in-depth discussion and several examples.

Working with the continuity equation in the form \eqref{eqn:modified_continuity_eqn} necessarily restricts our focus to measures $\rho$ with a differentiable density. Following \cite{chizat2018interpolating}, this constraint can be relaxed by working instead with a weak (i.e., distributional) form of the equation.

\begin{definition}[Distributional Continuity Equation]
Let $\rho_0,\rho_1 \in \Mplus(\Omega)$. We say that a triple $(\rho,\omega,\zeta)\in \Mplus([0,1]\times \Omega)\times \Meas([0,1]\times \Omega)^n\times \Meas([0,1]\times \Omega)$ satisfies the \define{continuity equation in the distributional sense} with boundary conditions $\rho_0$ and $\rho_1$ if, for any $\phi \in C^1([0,1]\times \Omega)$, 
\begin{equation}\label{eqn:dist_continuity_equation}
    \int_{[0,1]\times \Omega} \partial _t \phi d\rho + \int_{[0,1]\times \Omega} \nabla \phi \cdot d\omega  +\int_{[0,1]\times \Omega}\phi  d\zeta = \int_{\Omega}\phi(1,\cdot)d\rho_1 - \int_{\Omega}\phi(0,\cdot)d\rho_0 
\end{equation}
We denote by $\mathcal{CE}(\rho_0,\rho_1)$ the set of triples $(\rho,\omega,\zeta)$ satisfying the continuity equation with boundary conditions $\rho_0,\rho_1$.
\end{definition}

The unbalanced optimal transport distance of \cite{chizat2018interpolating} is then defined as follows. 

\begin{definition}[WFR distance]\label{def:wfr_distance}
    For $\delta>0$, the \define{Wasserstein-Fisher-Rao infinitesimal cost} is the function $f_\delta:\mathbb{R}\times \mathbb{R}^n\times \mathbb{R} \to [0,+\infty]$ on $\Mplus(\Omega)$ defined by 
    \begin{equation}
        f_{\delta}(\rho,\omega,\zeta) := \begin{cases}
            \frac{\|\omega\|^2+\delta^2\zeta^2}{2\rho}, & \rho>0 \\ 
            0, & (\rho,\omega,\zeta)=(0,0,0) \\
            +\infty, & \textrm{otherwise.} 
        \end{cases}
    \end{equation}
    The associated \define{Wasserstein-Fisher-Rao (WFR) distance} is then defined by
    \begin{equation}\label{eqn:WFR_distance}
    \mathrm{WFR}_\delta(\rho_0,\rho_1)^2 := \inf \left\{ \int_{[0,1]\times \Omega} f_\delta\left(\frac{d\mu}{d\lambda}\right)d \lambda \middle| 
    \begin{aligned}
         \mu = (\rho,\omega,\zeta)\in \mathcal{CE}(\rho_0,\rho_1) 
    \end{aligned}
    \right\},
\end{equation}
where $\lambda$ is any nonnegative measure on $[0,1]\times \Omega$ such that $\rho,\omega,\zeta\ll \lambda$ and $\frac{d\mu}{d\lambda} := (\frac{d\rho}{d\lambda}, \frac{d\omega}{d\lambda}, \frac{d\zeta}{d\lambda})$. We note that the value does not depend on the choice of $\lambda$ due to the 1-homogeneity of $f_{\delta}$. {To see this, we let $|\mu|$ be the total variation  of the vector measure $\mu=(\rho,\omega,\zeta)$. Then we have $|\mu| \ll \lambda$ and therefore
\begin{equation}
    \int_{[0,1]\times \Omega}f_\delta\left(\frac{d\mu}{d\lambda}\right) d\lambda = \int_{[0,1]\times \Omega}f_\delta\left(\frac{d\mu}{d|\mu| } \frac{d|\mu|}{d\lambda}\right) d\lambda = \int_{[0,1]\times \Omega}f_\delta\left(\frac{d\mu}{d|\mu| } \right) \frac{d|\mu|}{d\lambda}d\lambda = \int_{[0,1]\times \Omega}f_\delta\left(\frac{d\mu}{d|\mu| } \right) d|\mu|,
\end{equation}
which is independent of $\lambda$.}
\end{definition}

\begin{remark}
    The Wasserstein-Fisher-Rao infinitesimal cost belongs to a more general family of infinitesimal costs defined in \cite{CHIZAT20183090}; one can generalize the definition of WFR distance by replacing $f_\delta$ with another infinitesimal cost. We restrict our attention to that particular $f_\delta$ because it is natural in the sense that it interpolates between well-known distances on the space of measures: it is (a constant multiple of) the Wasserstein distance when $\delta = 0$ and $\frac{1}{\delta^2}\mathrm{WFR}_\delta$ limits to the Fisher-Rao distance as $\delta \rightarrow \infty$.
\end{remark}


Considering $\rho$ as an element of $\mathcal{M}^+([0,1] \times \Omega)$ apparently loses the dynamical interpretation of $\rho$ as a ``path of measures'', but this interpretation is recovered via the concept of disintegration. Recall that a \define{disintegration} of a measure $\nu$ on $[0,1] \times \Omega$ is a Borel family of measures $\{\nu_t\}_{t \in [0,1]}$ such that $d\nu(t,x) = dt \otimes d\nu_t(x)$. We will write in short $\nu = dt \otimes \nu_t$. Then the following holds:

\begin{proposition}
    \label{prop: disintegration in time}
     Let $(\rho,\omega,\zeta) \in\CE$. Then $\rho$ is disintegrable: $\rho=dt \otimes \rho_t$ for a Borel family of measures $\{\rho_t\}_{t}$. If we assume further that  $\omega \ll \rho$ (coordinate-wise) and $\zeta \ll \rho$ hold, we have the disintegration of $\omega$ and $\zeta$: $\omega = dt \otimes \omega_t, \zeta = dt\otimes \zeta_t$. The absolute continuity conditions hold whenever the objective of~\eqref{eqn:WFR_distance} is finite so that disintegrability holds in our setting without loss of generality.
\end{proposition}
\begin{proof}
This result can be deduced from the disintegration theorem found in~\cite[Theorem A.1]{Bredies_2020}, but we still write the proof for completeness. Let us first take a function $\phi:[0,1]\to \mathbb{R}$, which is $C^1$ and compactly supported on $(0,1)$, and define $\tilde{\phi} = \phi \circ \pi \in C^1([0,1]\times \Omega)$ where $\pi: [0,1]\times \Omega \to [0,1]$ is the projection onto $[0,1]$. Substitute $\tilde{\phi}$ to the continuity equation to obtain
\begin{equation}
    \int_{[0,1]\times \Omega}\partial_{t}\tilde{\phi} d\rho + \int_{[0,1]\times \Omega}\tilde{\phi} d\zeta = 0
\end{equation}
By the standard fact about pushforward measures, we have
\begin{equation}
    \int_{0}^{1}\partial_{t}\phi d(\pi_*\rho) + \int_{0}^{1}\phi d(\pi_*\zeta)= 0
\end{equation}
This equation implies that the distributional derivative of $\pi_* \rho$ equals $\pi_* \zeta$. Now, we define a function on $[0,1]$ by $u(t) = \int_{0}^{t}d (\pi_*\zeta)$. Then, for any open interval $A\subset [0,1]$, by Fubini's theorem,
\begin{align}
    \int_{A}u(t)dt &= \int_{t=0}^{t=1} 1_{A}(t)u(t)dt=\int_{t=0}^{t=1}\int_{s=0}^{s=1} 1_{A}(t)1_{[0,t]}(s)d(\pi_*\zeta)dt =\int_{s=0}^{s=1}\int_{t=0}^{t=1} 1_{A}(t)1_{[s,1]}(t)dtd(\pi_*\zeta)\\  &= \int_{s=0}^{s=1}\int_{t=s}^{t=1} 1_{A}(t)dtd(\pi_*\zeta)
\end{align}
The function $v(s) = \int_{t=s}^{t=1}1_{A}(t)dt$ is generally not a compactly supported $C^1((0,1))$ function but can be approximated by one since smooth bump functions can approximate the indicator. Therefore, by applying the fact that the distributional derivative of $\pi_*\rho$ is $\pi_* \zeta$,
\begin{align}
    \int_{s=0}^{s=1}\int_{t=s}^{t=1} 1_{A}(t)dtd(\pi_*\zeta) = \int_{s=0}^{s=1}1_{A}(s)d(\pi_*\rho) = \pi_*\rho(A)
\end{align}
Thus, $u$ is a density function of $\pi_*\rho$, and thus $\pi_* \rho$ is absolutely continuous with respect to the Lebesgue measure. Applying the disintegration theorem, we obtain the proposition.

\end{proof}

In light of this result, we frequently abuse terminology and refer to $\mu$, $\rho$, $\omega$, and $\zeta$ as \emph{paths of measures}, or simply as \emph{paths}.

\newenvironment{thisnote}{\par\color{blue}}{\par}

\section{The constrained Wasserstein-Fisher-Rao distance}
We will first focus on a version of constrained, unbalanced optimal transport, which leads to a constrained variant of the Wasserstein-Fisher-Rao distance \eqref{eqn:WFR_distance}. We start by defining the following constrained space of measures:
\begin{definition}[Constrained space] 
Given $\eich: \Omega\to\mathbb R^d$ and $\Ce\in \mathbb R^d$, for some integer $d \geq 1$, let $\MplushC$ be  the set  of all measures satisfying the affine constraint described by $\eich$ and $\Ce$, i.e.,  
\begin{equation}
    \MplushC := \left\{\rho \in \Mplus(\Omega)\middle| \int_{\Omega} \eich \; d\rho = \Ce\right\}.
\end{equation}
\end{definition}
Using this definition, it is straightforward to consider the corresponding constrained, unbalanced optimal transport problem:
\begin{definition}[Affine-constrained WFR problem]\label{def:WFR_problem_constant_time}
For a function $\eich:\Omega \to \R^d$ and a constant $C\in \R^d$, we say that a measure $\rho \in \Mplus([0,1] \times \Omega)$ \define{satisfies the affine constraint} with respect to $\eich$ and $C$ if it disintegrates in time ($\rho = dt \otimes \rho_t$) and 
$\rho_t \in \MplushC$ for all $t$. 
In this case, we write $\rho \in \AHC$. For $\rho_0, \rho_1\in \Mplus(\Omega)$, we consider the constrained set of measures
\[
\CEHC := \left\{\mu = (\rho,\omega,\zeta) \in \CE) \mid \rho \in \AHC\right\}.
\]
We define the  \textbf{affine-constrained Wasserstein-Fisher-Rao energy} with respect to $H$ and $C$ by
\begin{equation}
\label{eqn: constWFR_affine_time_constant}
\WFRconst\rhosq = \inf \left\{ \int_{[0,1]\times \Omega} f_\delta\left(\frac{d\mu}{d\lambda}\right)d \lambda \middle| 
    \begin{aligned}
         \mu = (\rho,\omega,\zeta)\in \CEHC
    \end{aligned}
    \right\},
\end{equation}
 where $\lambda$ is any nonnegative Borel measure on $[0,1]\times \Omega$ such that $\rho,\omega, \zeta \ll \lambda$. By the homogeneity of $f_\delta$, the integral in the definition does not depend on the choice of $\lambda.$ 
\end{definition}

We will later show that $\WFRconst$ defines a metric on $\MplushC$ under certain conditions on $\eich$ and $\Ce$ (Theorem \ref{thm:constrained_WFR_metric}). First, we will study the analytical properties of the optimization problem \eqref{eqn: constWFR_affine_time_constant}.


\subsection{Existence of solutions.}
The main results of this paper concern the existence of solutions for the constrained WFR optimization problems. Therefore{,} we will need the following notions of finite energy and minimal energy paths:
\begin{definition}
For $\rho_0,\rho_1 \in \Mplus(\Omega)$, we say that $\mu = (\rho,\omega,\zeta) \in \CEHC$ is a \define{finite energy path with respect to $\eich$ and $\Ce$} (or just a \define{finite energy path}) if the objective of \eqref{eqn: constWFR_affine_time_constant} is finite; that is, if $\int_{[0,1]\times \Omega} f_{\delta}\left(\frac{d\mu}{d\lambda}\right)d \lambda<\infty$. If $\mu$ realizes the infimum of \eqref{eqn: constWFR_affine_time_constant}, we say that $\mu$ is a \define{minimal energy path with respect to $\eich$ and $\Ce$} (or just a \define{minimal energy path}). 
\end{definition}
We now state our first main result.

\begin{theorem}[Existence of minimizers]\label{thm: AffWFRDuality}
    Let $\rho_0,\rho_1\in \mathcal{M}^+(\Omega)$, let $\eich:\Omega \to \mathbb{R}^d$ be a continuous function and let $\Ce\in \mathbb{R}^d$.  If a finite energy path with respect to $\eich$ and $\Ce$ exists, then a minimal energy path exists.
\end{theorem}
We will postpone the proof of this theorem to Section~\ref{sec:existence_of_minimizers_time_varying}, where we shall in fact prove a more general version of this result, where $\eich$ and $\Ce$ may also depend on the time $t$ (Theorem \ref{thm: AffWFRDuality_time}).

Our next goal is to prove the existence of finite energy paths (and consequently, by Theorem~\ref{thm: AffWFRDuality}, the existence of minimizers) assuming certain conditions on either $\eich$ or on the given measures $\rho_0$ and $\rho_1$:

\begin{theorem}[Existence of finite energy paths]
    \label{thm: Affh0 Finite Energy Path_time_constant}
    Let $\eich: \Omega \to \mathbb{R}^d$ be a continuous function and 
    let $\rho_0,\rho_1\in\Mplushf$. Assume, in addition, that one of the following holds:
    \begin{enumerate}
      \item $\Omega$ is convex, $d=1$, $\Ce>0$ and $\eich\in C^1(\Omega,[c,\infty))$ for some $c>0$.\label{ass:1}
    \item the constraint is linear, i.e., $\Ce=0$;\label{ass:2}
    \item there exists {a nonnegative Radon measure} $\bar \rho$ with respect to which the Radon-Nikodym derivative of both $\rho_0$ and $\rho_1$ is a continuous and positive function, i.e., ${\frac{d\rho_i}{d\bar{\rho}} }\in C(\Omega,\mathbb R_{>0})$, {where $\mathbb R_{>0}$ denotes the set of strictly positive real numbers.} \label{ass:3}
   
    \end{enumerate}
     Then there exists a finite energy path in $\CEHC$.
\end{theorem}

\begin{remark}\label{rem:continuity_time_constant}
    The continuity of $H$ in Theorem~\ref{thm: Affh0 Finite Energy Path_time_constant} is not actually necessary, and there is a finite energy path as long as the integral is defined. However, we only consider continuous functions here as this is the condition for $H$ in Theorem \ref{thm: AffWFRDuality}.
\end{remark}
\begin{proof}
Given assumption \eqref{ass:1} or assumption~\eqref{ass:2}, the proof follows directly from the corresponding more general results with time-varying constraints; see Theorems~\ref{thm:general_finite_energy_existence} and~\ref{prop:Affh0_Finite_Energy_Path} below, respectively. It remains to prove the result given that assumption~\eqref{ass:3} holds. Consider the path $(\rho,\omega,\zeta)$ with $\rho_t = (1-t) \rho_0 + t \rho_1$, $\omega$ identically zero and $\zeta = \partial_t \rho = \rho_1 - \rho_0$. This path clearly satisfies the affine constraint and the continuity equation. Moreover, the path is finite energy: since $\omega = 0$, {the energy is equal to the energy w.r.t. the Fisher-Rao metric} (that is, the objective function of the variational problem defining the distance \eqref{eqn:fisher_rao_distance}), which must be finite because the Fisher-Rao metric is a smooth, Riemannian metric on the space of all measures satisfying assumption~\eqref{ass:3}, i.e., in the space of all measures that have a continuous and strictly positive Radon-Nikodym derivative w.r.t. to some fixed measure $\bar \rho$ (see \cite{friedrich1991fisher,bauer2016uniqueness}). Note that this {is} not true {in} the space of all Radon-measures, and thus, the assumption is really needed, cf.~\cite[Section 1.2]{chizat2018interpolating}.   
\end{proof}


\subsection{The constrained WFR distance}\label{sec:constrained_WFR_distance}
In this subsection, we will show that the constrained WFR energy induces a metric structure on the set of all measures satisfying the given constraint. Furthermore, we will see that,
similar to the unconstrained setting, this metric structure can be realized as the geodesic distance of a (formal) Riemannian metric.

To this end, we will need the following formula, which is analogous to the unconstrained case in \cite{chizat2018interpolating} and \cite{Dolbeault_2008}:
\begin{lemma}
Let $\eich:\Omega \to \mathbb{R}^d$ and $\Ce\in \mathbb{R}^d$.
For any $\rho_0,\rho_1 \in \MplushC$, the following equivalent formulation of the constrained WFR energy holds: 
\begin{equation}
\label{eqn: equivalentWFR}
    \normalfont{\WFRconst(\rho_0, \rho_1)} = \inf \left\{ \int_{0}^{1}\left\{\int_{\Omega} f_{\delta}\left(\frac{d\mu_t}{d\lambda_t}\right)d \lambda_t \right\}^{1/2} dt \middle| 
\begin{aligned}
  \mu = (\rho,\omega,\zeta)\in \mathcal{CE}_{\eich,\Ce}(\rho_0,\rho_1)
\end{aligned}
\right\}
\end{equation}
where $\mu_t$ is the disintegration of $\mu$ in time. Here, $\lambda_t$ is any negative Borel measure so that $\mu_t \ll \lambda_t$.
Moreover, any minimal energy path $\mu$, solution to \eqref{eqn: constWFR_affine_time_constant}, satisfies
\begin{equation}
\label{eqn: constant_speed}
\normalfont{\WFRconst(\rho_0, \rho_1)^2=\int_{\Omega} f_{\delta}\left(\frac{d\mu_t}{d\lambda_t}\right)d \lambda_t }
\end{equation}
for a.e. $t\in [0,1]$.
\end{lemma}
We skip the proof for brevity as this is a direct adaptation of that of Theorem 5.4 in \cite{Dolbeault_2008}. This lemma now allows {us} to prove that the constrained WFR energy indeed defines a distance function on $\MplushC$. 

\begin{theorem}[WFR defines a metric]\label{thm:constrained_WFR_metric} Let $\eich:\Omega \to \mathbb{R}^d$ and $\Ce\in \mathbb{R}^d$ such that there is a finite energy path of the constrained WFR energy between any two measures in $\MplushC$. Then $\WFRconst$ defines a metric on $\MplushC$, and this metric can be interpreted as the geodesic distance of a (formal) Riemannian metric, which we will call the constrained WFR Riemannian metric. 
\end{theorem}

\begin{proof}
To show that $\normalfont{\WFRconst}$ is a distance function, we adapt the proof in \cite{charon_length_2021} to our situation.
For any $\rho_0,\rho_1 \in \MplushC$, we trivially have $\WFRconst(\rho_0,\rho_1)\geq 0$. If $\WFRconst(\rho_0,\rho_1) = 0$, then so is the unconstrained WFR distance between $\rho_0$ and $\rho_1$ and thus we have $\rho_0=\rho_1$. Conversely, if $\rho_0=\rho_1=\rho$, the constant path where both the momentum and the source term are zero shows that $\WFRconst(\rho,\rho)=0$ for any $\rho\in\Mplushf$.  

To show that $\WFRconst$ is symmetric, we observe that any finite energy path $\mu = (dt \otimes\rho_t, dt \otimes\omega_t, dt \otimes\zeta_t)$ from $\rho_0$ to $\rho_1$ in $\AHC$  can be reversed in time to construct a finite energy path from $\rho_1$ to $\rho_0$ still in $\AHC$: $\mu' = (dt\otimes\rho_{1-t}, -dt \otimes\omega_{1-t}, -dt \otimes\zeta_{1-t})$. The time reversal operation is a bijection on the space of finite energy paths as repeating the process yields the identity map. Moreover, the operation preserves the energy, so reversing a minimal energy path is, again, a minimal energy path. Thus, we have  $\WFRconst(\rho_0,\rho_1) =\WFRconst(\rho_1,\rho_0)$. 

It remains to prove the triangle inequality. Let $\rho_0,\rho_1,\rho_2\in \Mplushf$ and consider a minimal energy path $\mu^{01} = (dt \otimes\rho_t^{01}, dt \otimes\omega_t^{01}, dt \otimes\zeta^{01}_t)$ from $\rho_0$ to $\rho_1$ and a minimal energy path $\mu^{12}= (dt \otimes\rho_t^{12}, dt \otimes\omega_t^{12}, dt \otimes\zeta^{12}_t)$ from $\rho_1$ to $\rho_2$. Note that the existence of minimal energy paths is guaranteed by~Theorem~\ref{thm: AffWFRDuality}.
We then define $\mu^{02}=(dt \otimes\rho_t^{02}, dt \otimes\omega^{02}_t, dt\otimes\zeta_t^{02})$ by
\begin{align}
    \rho^{02}_t = 
    \begin{cases}
        \rho^{01}_{2t} & 0\leq t\leq 1/2 \\
        \rho^{12}_{2t-1} & 1/2 \leq t \leq 1
    \end{cases}, \ \
    \omega^{02}_t = 
    \begin{cases}
        2\omega^{01}_{2t} & 0\leq t\leq 1/2 \\
        2\omega^{12}_{2t-1} & 1/2 \leq t \leq 1
    \end{cases}, \ \
    \zeta^{02}_t = 
    \begin{cases}
        2\zeta^{01}_{2t} & 0\leq t\leq 1/2 \\
        2\zeta^{12}_{2t-1} & 1/2 \leq t \leq 1
    \end{cases} 
\end{align}
By direct calculation, we have $\mu^{02} \in \CEHC$. Using Equation \eqref{eqn: equivalentWFR},
we calculate
\begin{align}
    \WFRconst(\rho_0,\rho_2) &\leq   \int_{0}^{1/2}\left\{\int_{\Omega} f_{\delta}\left(\frac{d\mu_t^{02}}{d\lambda_t}\right)d \lambda_t \right\}^{1/2} dt +\int_{1/2}^{1}\left\{\int_{\Omega} f_{\delta}\left(\frac{d\mu_t^{02}}{d\lambda_t}\right)d \lambda_t \right\}^{1/2} dt \\& \leq \int_{0}^{1}\left\{\int_{\Omega} 4f_{\delta}\left(\frac{d\mu_{s}^{01}}{d\lambda_s}\right) d \lambda_s \right\}^{1/2} \frac{1}{2}ds +  \int_{0}^{1}\left\{\int_{\Omega} 4f_{\delta}\left(\frac{d\mu_{s'}^{12}}{d\lambda_{s'} }\right) d \lambda_{s'} \right\}^{1/2} \frac{1}{2}ds' \\ &= \left\{\int_{0}^{1}\int_{\Omega}f_{\delta}\left(\frac{d\mu_{s}^{01}}{d\lambda_s}\right) d \lambda_s  ds \right\}^{1/2} +\left\{\int_{0}^{1}\int_{\Omega}f_{\delta}\left(\frac{d\mu_{s'}^{12}}{d\lambda_{s'}}\right) d \lambda_{s'} ds' \right\}^{1/2} \\ &= \WFRconst(\rho_0,\rho_1) + \WFRconst(\rho_1,\rho_2)
\end{align}
Here, the first line is just the definition of the infimum, and the inequality in the second uses the fact that $\frac{\|2\omega\|^2+\delta^2(2\zeta)^2}{\rho} = 4\frac{\|\omega\|^2+\delta^2\zeta^2}{\rho}$ and the substitution $s=2t, s'=2t-1$.  The equality in the third line is just \eqref{eqn: constant_speed}. Finally, the last line follows from the fact that $\mu^{01}$ and $\mu^{12}$ are minimal energy paths.

Note that $\MplushC$ is an affine subspace and thus a submanifold of the infinite-dimensional manifold of all densities and that the definition of $\WFRconst$ coincides with the definition of the unconstrained WFR distance up to the fact that the set of allowable paths is constrained to this submanifold. Consequently, $\WFRconst$ is indeed the geodesic distance of a (formal) Riemannian metric, namely of the restriction of the Riemannian metric corresponding to the unconstrained WFR distance to the submanifold $\MplushC$. 
\end{proof}

Still on a formal level, one can further derive necessary optimality equations satisfied by constant-speed geodesics of the constrained WFR Riemannian metric. We will here restrict to measures $\rho,\omega,\zeta$ such that for all $t$, $\rho_t,\omega_t,\zeta_t$ are all positive smooth densities with respect to the Lebesgue measure $dx$ and also assume that the constraint function $H:\Omega\rightarrow \R^d$ is smooth. We will denote by $(H_i)_{i=1,\ldots,d}$ the coordinate functions of $H$. For our purpose, we first reformulate the constraints on $\rho$ as constraints on the momentum and source variables $\omega$ and $\zeta$ based on the following lemma:
\begin{lemma}
    \label{lem:affine_constraint_rho_zeta}
    Let $H = (H_1, \cdots , H_d) \in C^1(\Omega)^d$, $C = (C_1, \cdots C_d)\in \mathbb{R}^d$ and $\rho_0, \rho_1 \in \MplushC$. A finite energy path $(\rho, \omega, \zeta)$ that belongs to $\CE$ is in $\AFF(H, C)$ if and only if
    \begin{equation}
        \int_{\Omega}\nabla H_i\cdot d\omega_t + \int_{\Omega}H_i d\zeta_t = 0, \textrm{ a.e. } t\in [0,1] \textrm{ and }i=1,\cdots ,d \label{eqn:affine_constraint_rho_zeta}.
    \end{equation}
\end{lemma}
\begin{proof}
    For any $u\in C^1([0,1])$, since $H_i\in C^1(\Omega)$, we have $\phi=H _i(x)u_i(t)\in C^1([0,1]\times \Omega)$. Evaluating the continuity equation \eqref{eqn:dist_continuity_equation} with such a $\phi$ gives, by using the fact that $\rho_0, \rho_1 \in \MplushC$,
    \begin{equation}
        \int_{0}^1\left(\int_{\Omega}H_i(x)d\rho_t\right)  u'(t)dt + \int_{0}^1 u(t)\left(\int_{\Omega} \nabla H \cdot d\omega_t\right) dt + \int_{0}^1 u(t)\left(\int_{\Omega} H_i d\zeta_t\right) dt = C_i(u(1)-u(0)).
    \end{equation}
or equivalently
\begin{equation}
    \int_{0}^1\left(\int_{\Omega}H_i(x)d\rho_t - C_i\right)  u'(t)dt + \int_{0}^1 u(t)\left(\int_{\Omega} \nabla H_i \cdot d\omega_t + \int_{\Omega} H_i d\zeta_t\right) dt = 0
\end{equation}
Therefore, if $\int_{\Omega}H_i(x)d\rho_t = C_i$ for a.e. $t\in [0,1]$, since $u$ is arbitrary, we have $\int_{\Omega} \nabla H_i \cdot d\omega_t + \int_{\Omega} H d\zeta_t = 0$ for a.e. $t\in [0,1]$. Conversely, if $\int_{\Omega} \nabla H_i \cdot d\omega_t + \int_{\Omega} H_i d\zeta_t = 0$ for a.e. $t\in [0,1]$, again since $u$ is arbitrary, we have $\int_{\Omega}H_i(x)d\rho_t = C$ for a.e. $t\in [0,1]$.
\end{proof}

{
Based on  Lemma \ref{lem:affine_constraint_rho_zeta}, we introduce the following Lagrangian functional for the reformulated problem with constraints on $\omega$ and $\zeta$:
\begin{align*}
    \mathcal{L}(\rho,\omega,\zeta,\Phi,\gamma) = &\int_0^1 \int_{\Omega} \frac{\|\omega\|^2 + \delta^2 \zeta^2}{2\rho} dx dt + \int_0^1 \int_{\Omega} \Phi(t,x)(\partial_t \rho + \nabla \cdot \omega - \zeta) dx dt \\
    &\hspace{2in}+\sum_{i=1}^d \int_0^1 g_i(t) \int_{\Omega} (\nabla H_i \cdot \omega + H_i \zeta) dx dt,
\end{align*}
in which the two functions $(t,x) \mapsto \Phi(t,x)$ and $t \mapsto g_i(t)$ act as Lagrange multipliers for the different constraints. We thus get the following first-order optimality conditions:
\begin{align*}
    &0=D_\rho \mathcal{L} \cdot \delta \rho = \int_0^1 \int_\Omega \left(-\frac{\|\omega\|^2 + \delta^2 \zeta^2}{2 \rho^2} -\partial_t \Phi\right) \delta \rho \, dx dt \\
    &0=D_\omega \mathcal{L} \cdot \delta \omega = \int_0^1 \int_\Omega \left(\frac{\omega}{\rho} -\nabla \Phi+\sum_{i=1}^dg_i(t) \nabla H_i \right) \cdot \delta \omega \, dx dt \\
    &0=D_\zeta \mathcal{L} \cdot \delta \zeta = \int_0^1 \int_\Omega \left(\delta^2 \frac{\zeta}{\rho} - \Phi +\sum_{i=1}^d g_i(t) H_i\right)\delta\zeta \, dx dt 
\end{align*}
for all $\delta \rho, \ \delta\omega, \ \delta\zeta$. This leads to
\begin{align}
\label{eq:opt_cond1}
    \partial_t \Phi + \frac{\|\omega\|^2+\delta^2 \zeta^2}{2\rho^2} =0, \ \ \omega = \left(\nabla \Phi - \sum_{i=1}^d g_i(t) \nabla H_i \right) \rho, \ \ \zeta = \frac{1}{\delta^2}\left(\Phi-\sum_{i=1}^d g_i(t)H_i \right)\rho.
\end{align}
Using the above expressions of $\omega$ and $\zeta$ in \eqref{eqn:affine_constraint_rho_zeta} leads to the following $(d \times d)$ linear system on $g(t)=(g_i(t))_{i=1,\ldots,d}$ for a.e. $t\in[0,1]$:
\begin{equation}
\label{eq:expr_bar_phi}
    G_{\rho_t} g(t) = b_t, \ \text{with } G_{\rho_t} = \left(\langle H_i,H_j\rangle_{W^{1,2}_{\rho_t}}\right)_{i,j=1,\ldots d} \ \text{and } b_{t} = \left(\langle H_i,\Phi(t,\cdot) \rangle_{W^{1,2}_{\rho_t}} \right)_{i=1,\ldots,d},
\end{equation}
where $\langle u,v \rangle_{W^{1,2}_{\rho_t}} = \int_\Omega \left[\frac{1}{\delta^2} u v +\nabla u \cdot \nabla v \right] \rho_t dx $ for any two smooth functions $u,v$ on $\Omega$. We note that $G_{\rho_t}$ corresponds to the Gram matrix of the constraint functions $H_i$ with respect to the $W^{1,2}_{\rho_t}$ metric and is thus invertible as long as these are linearly independent. Thus we obtain $g(t) = G_{\rho_t}^{-1} b_t$ and given the definition of $G_{\rho_t}$ and $b_t$, we see that the vector $g(t) \in \R^d$ gives the coordinates of the projection of $\Phi$ onto the subspace spanned by the $H_i$'s. Consequently, $\bar{\Phi}(t,x)=\sum_{i=1}^{d} g_i(t) H_i(x)$ is exactly the orthogonal projection of $\Phi(t,\cdot)$ onto $\text{span}\{H_1,\ldots,H_d\}$ with respect to the $W^{1,2}_{\rho_t}$ Hilbert metric. Going back finally to \eqref{eq:opt_cond1}, we have obtained the following conditions on the optimal path:
\begin{equation}
\label{eq:opt_cond2}
\left\{
\begin{aligned}
&\partial_t \Phi + \frac{1}{2} \left\|\nabla \Phi(t,\cdot) - \nabla\bar{\Phi}(t,\cdot) \right\|^2 +\frac{1}{2\delta^2} \left(\Phi(t,\cdot) - \bar{\Phi}(t,\cdot)\right)^2 =0 
\\
&\partial_t \rho +\nabla \cdot(\rho (\nabla\Phi-\nabla \bar{\Phi})) = \frac{1}{\delta^2}\left(\Phi-\bar{\Phi}\right)\rho.
\end{aligned}
\right.
\end{equation}
These can be seen as a coupled system of PDEs on the density trajectory $\rho(t,x)$ and the potential function $\Phi(t,x)$ which, on the one hand, extends those of the standard Wasserstein-FR metric and, as we shall elaborate in Section \ref{ex:WFR_prob}, those of the spherical Hellinger-Kantorovich model derived in \cite{jing2024machine}. {In section \ref{sec:suff-cond-optimal}, we will show that this equation is, in fact, a sufficient condition for optimality, suggesting that the argument here has the potential to be extended beyond a formal argument.}

\subsection{Geometric aspects of the constrained WFR distance}
\label{ssec:Otto_picture}
Finally, we will study geometric aspects of the proposed constrained UOT model; namely, we aim to extend Otto's Riemannian submersion from optimal transport to the situation {of UOT with time-independent constraints}:
a classical result~\cite{OttoPic} in (balanced) optimal transport asserts that the flat $L^2$-metric on the group of diffeomorphisms of $\Omega$ descends to a Riemannian metric on the space of probability densities, whose geodesic distance function is exactly the Wasserstein distance. The simple nature of the geometry of the $L^2$-metric -- geodesics are given by straight lines -- gives rise to a beautiful geometric interpretation of the correspondence between the static and dynamic formulation of optimal transport{; see also~\cite{klas2017geometry} for an expository presentation of the geometric description of OT.} More recently, this picture has been extended to the unbalanced situation by Gallouet and Vialard~\cite{gallouet2018camassa}, see also~\cite{gallouet2021regularity}. Before we describe the extension of this geometric interpretation to the constrained model of the present article, we will briefly recall the constructions for the unconstrained situations. In the following, we will restrict ourselves to the category of smooth functions (densities, resp.), i.e., we consider the space of all smooth, positive densities 
\begin{equation}
\operatorname{Dens}(\Omega):=\left\{\psi\rho_0 \mid   \psi\in C^{\infty}(\Omega,\mathbb R_{>0}) \right\}, 
\end{equation} 
where $\rho_0$ is a fixed element of $\mathcal{M}^+(\Omega)$. We will also consider the subset of  
all smooth probability densities, 
\begin{equation}
\operatorname{Dens}_1(\Omega):=\left\{\psi\rho_0\in \operatorname{Dens}(\Omega) \mid   \int_{\Omega} \psi\rho_0=1 \right\}. 
\end{equation} 
Next we consider the group of smooth diffeomorphisms $\operatorname{Diff}(\Omega)$ and equip it with the (flat) $L^2(\rho_0)$-Riemannian metric $G^{L^2}$, i.e.,
\begin{equation}
G^{L^2}_{\varphi}(\delta\varphi,\delta\varphi)=\int_{\Omega} \langle \delta\varphi,\delta\varphi\rangle \rho_0, \qquad \text{ for }\varphi\in \operatorname{Diff}(\Omega),\quad \delta\varphi\in T_{\varphi}\operatorname{Diff}(\Omega).
\end{equation}
In this setup, Otto's Riemannian submersion result~\cite{OttoPic} asserts that \begin{align*} \pi: \operatorname{Diff}(\Omega) & \to \operatorname{Dens}_1(\Omega) \text{ given by } \\ \varphi & \mapsto \varphi_*\rho_0
\end{align*} is a formal Riemannian submersion from $\operatorname{Diff}(\Omega)$ with the (flat) $L^2(\rho_0)$ metric onto the space of smooth probability densities equipped with the (formal) Riemannian metric corresponding to the Wasserstein distance. 
To extend this construction to the constrained, unbalanced setting{,} we follow the work of Gallou\"{e}t and Vialard~\cite{gallouet2018camassa}: let $\mathcal{C}(\Omega)$ be the cone over $\Omega$  and consider the automorphism group
$\operatorname{Aut}(\mathcal{C}(\Omega))$, which can be viewed as the semi-direct product of $\operatorname{Diff}(\Omega)$ and $C^\infty(\Omega,\R^+)$. Next{,} we need to generalize the $L^2(\rho_0)$-metric to this setting. {Therefore we equip the cone $\mathcal{C}(\Omega)$ with the cone metric
\begin{equation}
\mathfrak{g}_{x,m}((dx,dm),(dx,dm)):= m dx^2+\delta^2 \frac{dm^2}{m},
\end{equation}
and let $\rho_0^{\mathcal{C}(\Omega)}$ be the induced volume form on the cone. 
This then leads to a natural definition of the corresponding $L^2(\rho_0)$-metric on $\operatorname{Aut}(\mathcal{C}(\Omega))$ given via
\begin{equation}\label{eq:coneL2metric}
G^{L^2(\rho_0)}_{\varphi,f}\big((\delta\varphi,\delta f),(\delta\varphi,\delta f)\big)=\int_{\mathcal{C}(\Omega)} \mathfrak{g}_{\varphi,f}((\delta\varphi,\delta f),(\delta\varphi,\delta f))\rho_0^{\mathcal{C}(\Omega)}
\end{equation}
}
In this setup{,} it is shown in ~\cite{gallouet2018camassa}, that the projection $\pi_0: \operatorname{Aut}(\mathcal{C}(\Omega)) \to \operatorname{Dens}(\Omega)$ given by $(\varphi,f) \mapsto \varphi_*(f^2\rho_0)$ is again a (formal) Riemannian submersion, where $\operatorname{Aut}(\mathcal{C}(\Omega))$ is equipped with the (flat) $L^2(\rho_0)$ Riemannian metric and $\operatorname{Dens}(\Omega)$ is equipped with the (formal) Riemannian metric corresponding to the Wasserstein-Fisher-Rao distance.

Mimicking these constructions, we obtain the following result for the constrained and unbalanced situation. For a detailed proof in the unconstrained situation, which can be readily adapted to the constrained situation, we refer to~\cite{gallouet2018camassa}.  
\begin{proposition}\label{prob:otto}
Let $\eich:\Omega \to \mathbb{R}^d$ and $\Ce\in \mathbb{R}^d$. Denote by $\operatorname{Dens}_{\eich,\Ce}(\Omega)=\operatorname{Dens}(\Omega)\cap\MplushC$ the set of smooth densities satisfying the constraint prescribed by $\eich$ and $\Ce$ and by $\operatorname{Aut}_{\eich,\Ce}(\mathcal{C}(\Omega))$ the corresponding subset of $\operatorname{Aut}(\mathcal{C}(\Omega))$ such that $\pi_0(\operatorname{Aut}_{\eich,\Ce}(\mathcal{C}(\Omega)))=\operatorname{Dens}_{\eich,\Ce}(\Omega)$. 
Then the projection
\begin{align*} \pi_0: \operatorname{Aut}_{\eich,\Ce}(\mathcal{C}(\Omega)) & \to \operatorname{Dens}_{\eich,\Ce}(\Omega) \text{ given by } \\ (\varphi,f) & \mapsto \varphi_*(f^2\rho_0)
\end{align*} 
 is a (formal) Riemannian submersion, where $\operatorname{Aut}_{\eich,\Ce}(\mathcal{C}(N))$ is equipped with the restriction of the (flat) $L^2(\rho_0)$ Riemannian metric{~\eqref{eq:coneL2metric}} and $\operatorname{Dens}_{\eich,\Ce}(\Omega)$ is equipped with the restriction of the (formal) Riemannian metric corresponding to the Wasserstein-Fisher-Rao distance.
\end{proposition}
The above result allows us to introduce a static (Monge) formulation of the constrained WFR metric:
\begin{definition}[Monge-formulation]
Let $\eich:\Omega \to \mathbb{R}^d$ and $\Ce\in \mathbb{R}^d$. For $\rho_0,\rho_1\in \Dens_{\eich,\Ce}(\Omega)$ the Monge formulation of the constrained WFR is given by
\begin{equation}
\operatorname{M}\text{-}\WFRAF(\rho_0,\rho_1)=\inf_{\varphi,f}\left\{d_{\operatorname{Aut}_{\eich,\Ce}}\big((\operatorname{id},1),(\varphi,f)\big): \varphi_*(f^2 \rho_0)=\rho_1\right\},
\end{equation}
where $d_{\operatorname{Aut}_{\eich,\Ce}}$ denotes the geodesic distance of the $L^2$-metric on the automorphism group restricted to the submanifold $\operatorname{Aut}_{\eich,\Ce}(\mathcal{C}(N))$. 
\end{definition}
It is clear that we have $\operatorname{M}\text{-}\WFRAF(\rho_0,\rho_1)\geq \WFRAF(\rho_0,\rho_1)$ as the corresponding infimum for the constrained WFR distance is taken over a much larger space. For the unconstrained version, it has been shown in~\cite{gallouet2021regularity}, that under certain assumption{s} on $\rho_0$ and $\rho_1$ and the underlying domain $\Omega$, the two formulations are indeed equivalent, i.e., that one has 
$\operatorname{M}\text{-}\WFRAF(\rho_0,\rho_1)=\WFRAF(\rho_0,\rho_1)$. In future work{,} it would be interesting to obtain a similar statement for the constrained version, but this seems difficult as the induced geodesic distance on the constrained automorphism group is significantly more complicated than in the unconstrained case, where this reduces to the standard $L^2$-metric on the cone.

\subsection{Case of probability densities: connection with the spherical Hellinger-Kantorovich metric}\label{ex:WFR_prob}
We conclude this section by specifying the above result to one important particular case, namely the restriction of the WFR metric to the space of probability densities. In the following, let $\operatorname{Prob}(\Omega)$ denote the space of all probability densities. The restriction of the WFR metric onto this space {was} first studied by Laschos and Mielke in~\cite{laschos2019geometric}, where they proved that the space of all densities equipped with the Wasserstein-Fisher-Rao metric can be viewed as a conic extension over the space of probability densities, with the restriction of this (formal) Riemannian metric. This in turn allowed them to derive an explicit formula for the geodesic distance of the Wasserstein-Fisher-Rao metric on the space of proba{b}ility densities, which they refer to as the \emph{spherical Hellinger-Kantorovich distance}. 

We first observe that $\operatorname{Prob}(\Omega)=\operatorname{Dens}_{1,1}(\Omega)$, i.e., this case fits the present article's setting by choosing $H(x)=1$ and $C=1$. Perhaps the most direct way to evidence the connection with the spherical Hellinger-Kantorovich model is to examine the geodesic equations \eqref{eq:opt_cond2} for the specific metric $\operatorname{WFR}^{1,1}_\delta$. In that case, since $d=1$ and $H=1$, $\nabla H =0$ and one finds $g(t) = \int_{\Omega} \Phi(t,x) \rho_t(x) dx$ from which it follows that $\bar{\Phi}(t,\cdot)$ is here the constant function of $x$ given by the expected value of $\Phi$ on $\Omega$ with respect to the density $\rho_t$. Then \eqref{eq:opt_cond2} becomes:
\begin{equation*}
\left\{
\begin{aligned}
&\partial_t \Phi + \frac{1}{2} \|\nabla \Phi\|^2 +\frac{1}{2\delta^2} \left(\Phi - \bar{\Phi}\right)^2 =0 \\
&\partial_t \rho +\nabla \cdot(\rho \nabla \Phi) = \frac{1}{\delta^2}\left(\Phi-\bar{\Phi}\right)\rho
\end{aligned}
\right.
\end{equation*}
which precisely matches, up to changing $\Phi$ to $-\Phi$, the geodesic equations of the spherical Hellinger-Kantorovich metric established in \cite{jing2024machine} (Section 3.2). Therefore, the result of~\cite{laschos2019geometric} can be reframed within our framework as follows:
\begin{theorem}[Theorem 3.4 of~\cite{laschos2019geometric}]
The space $(\Dens,\operatorname{WFR}_{\delta})$ of positive densities equipped with the WFR distance can be identified with the cone over the spherical space $\operatorname{Prob}(M)$ equipped with the distance corresponding to the constrained, unbalanced metric $\operatorname{WFR}^{1,1}_\delta$.
\end{theorem}
To the best of our knowledge{,} the analogue of Otto's submersion picture, as obtained in Section~\ref{ssec:Otto_picture}, is new even for the specific situation of unbalanced OMT constrained to the space of probability densities. In the following{,} we will briefly discuss Proposition \ref{prob:otto} for this special case, where it remains to investigate the constrained space of automorphisms $\operatorname{Aut}_{1,1}(\mathcal{C}(\Omega))$.
\begin{corollary}[Otto's Riemannian submersion for the spherical Hellinger-Kantorovich distance]
Let $\mathcal S_1$ be the $L^2(\rho_0)$ unit sphere in $C^{\infty}(\Omega)$.
Then the projection
\begin{align*} \pi_0: \operatorname{Diff}(\Omega)\ltimes\mathcal S_1 & \to \operatorname{Prob}(\Omega) \text{ given by } \\ (\varphi,f) & \mapsto \varphi_*(f^2\rho_0)
\end{align*} 
 is a (formal) Riemannian submersion, where $\operatorname{Diff}(\Omega)\ltimes\mathcal S_1$ is equipped with the restriction of the (flat) $L^2(\rho_0)$ Riemannian metric{~\eqref{eq:coneL2metric}} and $\operatorname{Prob}(\Omega)$ is equipped with the (formal) Riemannian metric corresponding to the spherical Hellinger-Kantorovich distance.
\end{corollary}
\begin{proof}
To see that this result follows from Proposition \ref{prob:otto}, we observe that the action of the group of diffeomorphisms on the space of densities preserves the total volume. Thus for $\rho_0\in\operatorname{Prob}(\Omega)$ we have that $\varphi_*(f^2\rho_0)\in \operatorname{Prob}(\Omega)$ if and only if $\int_{\Omega}f^2\rho_0=1$, i.e., if $f$ is an element of the $L^2(\rho_0)$ unit sphere $\mathcal S_1$. Consequently, we can represent the constrained space as the semidirect product $\operatorname{Aut}_{1,1}(\mathcal{C}(\Omega))=\operatorname{Diff}(\Omega)\ltimes\mathcal S_1$. 
\end{proof}

\begin{example}[Higher order moment constraints]
\label{ex:WFR_area_measures}
Beyond imposing constraints on the total mass, one may also consider situations in which some higher order moments may be fixed, for instance the expected value or variance. In the former case, restricting to measures with a prescribed expected value $m \in \mathbb{R}^d$ corresponds to a constraint of the form $\int_\Omega x d\rho(x) = m$ which exactly fits the model of this section by setting $H(x) = x$ and $C=m$. More generally, setting $H(x) = (x_{i_1} - c_1) \ldots (x_{i_n}-c_n)$ for $n \in \mathbb{N}$, $\{i_1,\ldots,i_n\}$ a set of indices in $\{1,\ldots,d\}$ and $c_1,\ldots,c_n \in \mathbb{R}$ allows to constrain any specific moment of order $n$ to a given value. The results of this section show that the corresponding metrics $\WFRAF$ can be used to compare and interpolate between measure distributions in these constrained subspaces (assuming finite energy paths exist, notably under the conditions of Theorem \ref{thm: Affh0 Finite Energy Path_time_constant}).    
\end{example}

\section{Unbalanced optimal transport with time-dependent constraints}
Having previously examined the case of constraints which are stationary in time, we seek to extend this analysis to time-dependent ones. We should point out right away that, in this case, one cannot expect the metric setting of Section \ref{sec:constrained_WFR_distance} to generalize and our primary focus will be on the issue of existence of solutions to the constrained unbalanced optimal transport problem. We introduce the following variant of the constrained, unbalanced optimal transport energy, which obviously generalizes the problem in Definition \ref{def:WFR_problem_constant_time}.

\begin{definition}[WFR problem with time-varying constraints]\label{def:wfr_problem_time_varying}
For functions $\eich:[0,1] \times \Omega \to \R^d$ and $\ef:[0,1] \to \R^d$, we say that a measure $\rho \in \Mplus([0,1] \times \Omega)$ \define{satisfies an affine constraint} with respect to $\eich$ and $\ef$ if it disintegrates in time ($\rho = dt \otimes \rho_t$) and 
\begin{equation}
\int_\Omega \Htx d\rho_t(x) = \Ft \qquad \mbox{for a.e. $t \in [0,1]$}.
\end{equation}
In this case, we write $\rho \in \AHF$. For $\rho_0, \rho_1\in \Mplus(\Omega)$, we consider the constrained set of measures
\[
\CEHF:= \left\{\mu = (\rho,\omega,\zeta) \in \mathcal{CE}(\rho_0,\rho_1) \mid \rho \in \AHF\right\}.
\]
We define the  \textbf{Wasserstein-Fisher-Rao energy with time-varying affine-constraints} by
\begin{equation}
\label{eqn: constWFR_affine}
\WFRAF\rhosq = \inf \left\{ \int_{[0,1]\times \Omega} f_\delta\left(\frac{d\mu}{d\lambda}\right)d \lambda \middle| 
    \begin{aligned}
         \mu = (\rho,\omega,\zeta)\in \CEHF
    \end{aligned}
    \right\},
\end{equation}
 where $\lambda$ is any nonnegative Borel measure on $[0,1]\times \Omega$ such that $\rho,\omega, \zeta \ll \lambda$. By the homogeneity of $f_\delta$, the integral in the definition does not depend on the choice of $\lambda.$ 
\end{definition}
\begin{example}[Total mass constraints and moving barrier]\label{ex:totalmass_movingbarrier}
Here we want to describe two specific examples, that can be formulated in our setting. The first one,  time dependent total mass constraints, is a direct extension of {Section}~\ref{ex:WFR_prob} to the time dynamic situation, i.e., $\eich=1$ but $\ef$ can be an arbitrary time dependent function that describes the evolution of the total mass. Such a situation naturally occurs in applications, e.g., when modeling a population density over time, where one only observes the total population but not its spatial distribution at intermediate times. Under mild conditions on $\ef$ we will show the existence of minimizers for this model in Theorems~\ref{thm: AffWFRDuality_time} and~\ref{prop:Affh0_Finite_Energy_Path}. 

As a second example, we aim to consider the situation of a domain $\Omega$ with a time-dependent barrier, i.e., a region $\Gamma(t)\subset \Omega$ through which no mass can be transported or created. Assuming that for each $t \in [0,1]$, $\Gamma(t)$ is an open subdomain of $\Omega$ that varies continuously in time, we can formulate this in our setting by choosing $\ef(t)=0$ and $\eich:[0,1] \times \Omega$ a contin{u}ous function with $\eich(t,x)>0$ for $x \in \Gamma(t)$ and $H(t,x)=0$ if $x \in \Omega \setminus \Gamma(t)$. It is indeed clear that the condition $\int_{\Omega} H(t,x) d \rho_t(x) =0$ is then equivalent to $\rho_t(\Gamma(t))=0$, in other words that no mass appears in $\Gamma(t)$. The existence of minimizers for such a model is again guaranteed by Theorems~\ref{thm: AffWFRDuality_time} and~\ref{prop:Affh0_Finite_Energy_Path} below as long as the initial and the terminal measures $\rho_0$ and $\rho_1$ stays inside the barrier at almost all time, i.e., $\rho_0(\Gamma(t))=\rho_1(\Gamma(t))=0$ for a.e. $t$. 
\end{example}

\subsection{Existence of minimizers of the WFR energy with time-varying constraints}\label{sec:existence_of_minimizers_time_varying}
We are now able to state our general results on {the} existence of minimizers. The finite and minimal energy path terminology is defined in analogy with the previous section.

\begin{theorem}[Existence of minimizers for time-varying affine-constraints]
    \label{thm: AffWFRDuality_time}    Let $\rho_0,\rho_1\in\Mplus(\Omega)$, let $\eich:[0,1]\times \Omega \to \mathbb{R}^d$ be a continuous function and let $\ef=(\ef_1,\ldots,\ef_d):[0,1]\to \mathbb{R}^d$ be a function with integrable component functions $\ef_i$. If a finite energy path with respect to $\eich$ and $\ef$ exists, then a minimal energy path exists. 
\end{theorem}

Note that this theorem also implies the corresponding result for stationary constraints, cf. Theorem~\ref{thm: AffWFRDuality}. The proof of the theorem uses the Fenchel-Rockafellar theorem from convex analysis, which we state here for the convenience of the reader.
\begin{theorem}[Fenchel-Rockafellar strong duality~\cite{borwein2005techniques}]
\label{thm:Fenchel_Rockafellar}
Let $X, Y$ be Banach spaces, let ${\mathcal G}:X\to \mathbb{R}\cup \{+\infty\}, {\mathcal F}:Y\to \mathbb{R}\cup \{+\infty\}$ be convex functions and let $A:X\to Y$ be a bounded linear map. If there exists $x_0\in X$ such that $\mathcal{G}(x_0)<\infty, \mathcal{F}(Ax_0)<\infty $ and $\mathcal{F}$ is continuous at $Ax_0$, we have
\begin{equation}
    \inf_{x\in X}\{{\mathcal F}(Ax)+{\mathcal G}(x)\} = \sup_{y^*\in Y^*}\{-{\mathcal F}^*(y^*)-{\mathcal G}^*(-A^*y^*)\}
\end{equation}
and there exists a $y^*\in Y^*$ which attains the sup on the right-hand side if it is finite.
\end{theorem}
\begin{proof}[Proof of Theorem \ref{thm: AffWFRDuality_time}]
    To apply the Fenchel-Rockafellar theorem, we introduce the primal problem
    \begin{equation}
    \inf_{x=(\phi,\psi)\in X}\{{\mathcal F}(Ax)+{\mathcal G}(x)\},
\end{equation}
with the following spaces and functionals:
\begin{itemize}[leftmargin=*]
    \item The Banach spaces $X$ and $Y$ are 
    \[
    X= C^1([0,1]\times \Omega)\times C([0,1])^d \quad \mbox{and} \quad Y= C([0,1]\times \Omega)\times C([0,1]\times \Omega)^n\times C([0,1]\times \Omega),
    \]
    respectively.
    \item The convex function ${\mathcal F}: Y\to \mathbb{R}\cup \{+\infty\}$ is  defined by
    \[
    {\mathcal F}(\alpha,\beta,\gamma) = \int_{[0,1]\times \Omega}\iota_{B_\delta}(\alpha(t,\theta),\beta(t,\theta),\gamma(t,\theta))dtd\theta,
    \]
    where 
    \begin{equation}
        B_{\delta} = \left\{(a,b,c)\in \mathbb{R}\times \mathbb{R}^d\times \mathbb{R} \middle| a+\frac{1}{2}\left(\|b\|^2+\frac{c^2}{\delta^2}\right) \leq 0\right\}
    \end{equation}
    and $\iota_{B_\delta}$ is the convex indicator function of $B_\delta$; that is,\\
    \begin{equation}
         \iota_{B_\delta}(a,b,c)  = \begin{cases}
                                        0 & (x,y,z)\in B_{\delta} \\
                                        +\infty & \textrm{otherwise.}
                                    \end{cases}
    \end{equation}
    Note that the Fenchel conjugate of $\iota_{B_\delta}$ is the Wasserstein-Fisher-Rao infinitesimal cost $f_\delta$ used to define $\WFRAF$.

    \item The convex function ${\mathcal G}:X\to \mathbb{R}\cup \{+\infty\}$ is defined by
    \[
    {\mathcal G}(\phi,\psi) = -\sum_{i=1}^{d}\int_{0}^{1}\psi_i(t)\ef_i(t)dt+\int_{\Omega}\phi(0,\cdot)d\rho_0 - \int_{\Omega}\phi(1,\cdot)d\rho_1.
    \]
    where $\psi_i$ and {$\ef_i$} are the $i$-th component of $\psi$ and {$\ef$}, respectively.
    \item The bounded linear map $A:X\to Y$ is defined by 
    \[
    A(\phi,\psi) = \left(\partial_{t}\phi+\sum_{i=1}^{d}\eich_i\psi_i, \nabla\phi, \phi\right).
    \]
\end{itemize}
Theorem \ref{thm:Fenchel_Rockafellar} indeed applies to this problem as $x_0=(\phi_0,\psi_0)$, where $\phi_0(t) = -\delta^2t$ and $\psi_0(t) = (0,\cdots ,0)$, satisfies the necessary conditions. Therefore, 
\begin{equation}\label{eqn:fenchel_rockafellar_equation}
    \inf_{x\in X}\{{\mathcal F}(Ax)+{\mathcal G}(x)\} = \sup_{y^*\in Y^*}\{-{\mathcal F}^*(y^*)-{\mathcal G}^*(-A^*y^*)\}
\end{equation}
By the Riesz-Markov-Kakutani Theorem \cite[Theorem 7.17]{folland2013real}, we can identify 
\[
Y^* \approx \mathcal{M}([0,1]\times \Omega)\times\mathcal{M}([0,1]\times \Omega)^n\times  \mathcal{M}([0,1]\times \Omega).
\]
Our goal is to calculate each term on the right-hand side inside the supremum in \eqref{eqn:fenchel_rockafellar_equation}. We start with the $G^\ast$ term. In the following, we always use $\langle\cdot,\cdot\rangle$ to denote the canonical pairing of a vector space with its dual. We have
\begin{align}
    {\mathcal G}^*(-A^*y^*) &= \sup_{(\phi,\psi)=x\in X}\{\langle -A^*y^*,x\rangle-{\mathcal G}(x)\} = \sup_{x\in X}\left\{-\langle y^*,Ax\rangle-{\mathcal G}(x)\right\} \\ 
    &=\sup_{x\in X}\left\{-\langle y^*,Ax\rangle+\sum_{i=1}^{d}\int_{0}^{1}\psi_i(t)\ef_i(t)dt+\int_{\Omega}\phi(1,\cdot)d\rho_1-\int_{\Omega}\phi(0,\cdot)d\rho_0\right\}.
\end{align}
Using the notation $y^*=(\rho,m,\zeta)$, we have, for $x = (\phi,\psi)\in X$,
\begin{align}
    \langle y^*,Ax\rangle &= \langle y^*, A(\phi,\psi)\rangle = \langle y^*,(\partial_{t}\phi+\textstyle \sum_{i=1}^{d}\eich_i\psi_i, \nabla\phi, \phi)\rangle \\ &= \int_{[0,1]\times \Omega} \partial_{t}\phi d\rho + \int_{[0,1]\times \Omega}\nabla\phi \cdot d\omega +\int_{[0,1]\times \Omega}\phi d\zeta +\sum_{i=1}^{d}\int_{[0,1]\times \Omega} \psi_i(t)\eich_i(t,x)d\rho \\
    &= \int_{[0,1]\times \Omega} \partial_{t}\phi d\rho + \int_{[0,1]\times \Omega}\nabla\phi \cdot d\omega +\int_{[0,1]\times \Omega}\phi d\zeta +\sum_{i=1}^{d}\int_{0}^{1}\psi_i(t)\int_{\Omega}\eich_i(t,x)d\rho_t dt.   
\end{align}
The calculations so far show that ${\mathcal G}^*(-A^*y)$ equals to
\begin{align*}
    &\sup_{\phi}\left\{-\int_{[0,1]\times \Omega} \partial_{t}\phi d\rho - \int_{[0,1]\times \Omega}\nabla\phi \cdot d\omega -\int_{[0,1]\times \Omega}\phi d\zeta +\int_{\Omega}\phi(1,\cdot)d\rho_1 -\int_{\Omega}\phi(0,\cdot)d\rho_0 \right\} \\
    &\hspace{2in} +\sum_{i=1}^{d}\sup_{\psi_i}\left\{\int_{0}^{1} \left(\ef_i(t)-\int_{\Omega}\eich_i(t,x)d\rho_t(x)\right)\psi_i(t)dt\right\}
\end{align*}
and it follows that 
\[
{\mathcal G}^*(-A^*y) = \left\{\begin{array}{rl}
0 & \mbox{if } (\rho, \omega, \zeta)\in \CEHF \\
+\infty & \mbox{otherwise.}
\end{array}\right.
\]
This implies that
\begin{equation}
    \sup_{y^*\in Y^*}\{-{\mathcal F}^*(y^*)-{\mathcal G}^*(-A^*y^*)\} = -\inf\left\{{\mathcal F}^*(y^*)\middle|\begin{aligned}
     y^* = (\rho,\omega, \zeta) \in \CEHF
\end{aligned}\right\},
\end{equation}
 so the existence of a minimal energy path follows by showing that $F^*(y^*)$ matches the objective of \eqref{eqn: constWFR_affine}. This step is done in the same way as in \cite{chizat2018interpolating}, using \cite[Theorem 5]{rockafellar_integrals_1971}. Following that, since the infimum is finite, we have that there exists $y^*\in \CE$ that attains the infimum by Fenchel-Rockafeller.
\end{proof}

\begin{remark}
It is worth pointing out that the statement and proof of Theorem \ref{thm: AffWFRDuality_time} could be easily adapted to constrained balanced optimal transport as well{,} i.e.{,} to the setting in which $\rho_0$ and $\rho_1$ are in addition assumed to be probability measures and only transport is considered (i.e. $\zeta =0$), assuming again that a finite energy path of probability measures exist. Such a model with a very specific constraint was analyzed in \cite{charon_length_2021}.   
\end{remark}

\subsection{Existence of finite energy paths for the WFR energy with time-varying constraints.}
According to Theorem \ref{thm: AffWFRDuality_time}, to guarantee the existence of a minimal energy path between two measures, it suffices to find a finite energy path. As we shall see, unlike the original unconstrained model, this turns out to be a non-trivial question {and we are unfortunately unable to answer this question in full generality. }

Let us first evacuate a few straightforward cases:
\begin{itemize}[leftmargin=*]
    \item If $\eich=0$, then there is no feasible solution unless $\ef=0$ for a.e. $t\in[0,1]$.
    \item If $\ef=\eich=0$, then the problem is unconstrained (i.e., $\AFF(0,0) = \Mplus([0,1] \times \Omega)$), so there is a finite energy path, as shown in \cite{chizat2018interpolating}. 
\end{itemize}
Our next goal is to examine the existence of finite energy paths for various relevant choices of $\eich$ and $\ef$. The main results are contained in the following theorems {(Theorems~\ref{prop:Affh0_Finite_Energy_Path} and~\ref{thm:general_finite_energy_existence}); the first one considers vector-valued constraints but assumes them to be linear (i.e. $\ef=0$), whereas the second considers 
general affine constraints but restricts them to be scalar-valued.} 
\begin{theorem}[Finite energy paths for linear vector valued constraints]
    \label{prop:Affh0_Finite_Energy_Path}
    Let $\eich:[0,1]\times \Omega \to \mathbb{R}^d$ be a continuous function and let $\rho_0,\rho_1 \in \mathcal{M}^+(\Omega)$ satisfy $\int_{\Omega}\Htx d\rho_0(x) = 0$ and $\int_{\Omega}\Htx d\rho_1(x) = 0$ for a.e. $t\in [0,1]$. Then there exists a finite energy path in $\mathcal{CE}_{\eich,0}(\rho_0,\rho_1)$ { and we have the upper bound
  \begin{equation}
    \WFRAFho\rhosq \leq 4\delta^2 (\rho_0(\Omega)+\rho_1(\Omega)).
    \end{equation}}
\end{theorem}
\begin{proof}
Similar to the proof of \cite[Theorem 1]{kondratyev2016new}, define $\mu=(\rho,m,\zeta)=(dt \otimes \rho_t, dt \otimes \omega_t, dt \otimes \zeta_t)$ by
    \begin{equation}
        \rho_t = 
        \begin{cases}
            (1-2t)^2\rho_0, & 0\leq t\leq 1/2 \\
            (2t-1)^2\rho_1 & 1/2 < t \leq 1 
        \end{cases}, \quad
        \omega_t = 0, \quad \mbox{and} \quad
        \zeta_t =
        \begin{cases}
            -4(1-2t)\rho_0, & 0\leq t\leq 1/2 \\
            4(2t-1)\rho_1 & 1/2 < t \leq 1.
        \end{cases}
    \end{equation}
We note that $\rho$ is a concatenation of the Fisher-Rao geodesic (\cite{chizat2018interpolating}, Theorem 3.1) from $\rho_0$ to $0$ with the one from $0$ to $\rho_1$. By direct calculation, it is easy to show that $\mu$ satisfies the constraint and the continuity equation. We now calculate the integral. Define $\lambda = \rho$. Clearly, $\rho, \omega, \zeta \ll \lambda$, and we have
\begin{equation}
    d\rho/d\lambda = 1, \quad
    d\omega/d\lambda = 0, \quad \mbox{and } \quad
    d\zeta/d\lambda = 
    \begin{cases}
        -\frac{4}{1-2t} & 0\leq t < 1/2 \\ 
        \frac{4}{2t-1} & 1/2 < t \leq 1.
    \end{cases}
\end{equation}
Therefore, we obtain
\begin{align}
    &\int_{[0,1]\times \Omega} f_{\delta}\left(\frac{d\mu}{d\lambda}\right)d \lambda = \int_{0}^{1} \int_{\Omega} \frac{0^2 + \delta^2(d\zeta/d\rho)^2}{2\cdot 1}d\rho_t dt\\ &= \frac{\delta^2}{2}\int_{0}^{1/2} \int_{\Omega} \left(\frac{-4}{1-2t}\right)^2 (1-2t)^2d\rho_0dt \nonumber +\frac{\delta^2}{2}\int_{1/2}^{1} \int_{\Omega} \left(\frac{4}{2t-1}\right)^2 (2t-1)^2d\rho_1dt = 4\delta^2(\rho_0(\Omega)+\rho_1(\Omega))<\infty.
\end{align}
\end{proof}  
{Next, we will present the existence of finite energy paths for affine scalar constraints:}
\begin{theorem}[Finite energy paths for affine scalar constraints]
\label{thm:general_finite_energy_existence}
    Suppose $\Omega$ is convex. Let $\eich:[0,1]\times \Omega \to \R$ be a differentiable function with bounded partial derivatives such that $\eich(t,x) > 0$ for all $(t,x) \in [0,1] \times \Omega$. {Let $\ef\in H^1([0,1])$, such that $\ef(t) \geq 0$ for all $t \in [0,1]$ and such that 
    \begin{equation}\label{int_assumption}
        \int_0^1 \frac{F'(t)^2}{F(t)}dt< \infty.
    \end{equation}
    Let $\rho_0,\rho_1$ be measures such that $\int_{\Omega}\eich(0,x)d\rho_0 = \ef(0)$ and $\int_{\Omega}\eich(1,x)d\rho_1 = \ef(1)$. Then, there exists a finite energy path in $\CEHF$. }
\end{theorem}

\begin{remark}[Decay Conditions for $\ef$]
{In this remark we will explore  assumption~\eqref{int_assumption} in more detail. We first note that this is clearly satisfied if $\ef(t)>0$ for all $t\in[0,1]$; using that the supremum norm of $\frac{1}{\ef}$ is bounded, this follows from the assumption that $\ef'$ is square integrable. }

Moreover assumption~\eqref{int_assumption} allows for constraint function $\ef$ to have zeros, but requires a certain decay towards these zeros. Therefore we introduce the following notation: for a function $\psi(t)$ we write $\psi(t)=\Theta(t^\alpha)$ as $t\to 0^+$ for some $\alpha>0$ if there exist $K_1,K_2,\epsilon>0$, such that $K_1 t^\alpha \leq \psi(t)\leq K_2t^\alpha$ for all $t \leq \epsilon$. It is easy to show that $\ef$ satisfies the assumption if the following two conditions are verified
\begin{itemize}
    \item $\ef(t) = 0$ only on finitely many subintervals $[a_i,b_i]$, $i=1,\cdots,n$, with $0\leq a_1 < b_1 < a_2 < b_2 < \cdots < a_n < b_n \leq 1$; and
    \item  The derivative $\ef'$ satisfies $\ef'(t)=\Theta((a_i- t)^{\alpha_i})$ as $t\to a_i^-$ and $\ef'(t)=\Theta((t-b_i)^{\beta_i})$ as $t\to b_i^+$ for some choices of $\alpha_i,\beta_i>0$.
\end{itemize}
\end{remark}

{Observe that Theorem \ref{prop:Affh0_Finite_Energy_Path} and Theorem \ref{thm:general_finite_energy_existence} imply cases \eqref{ass:1} and \eqref{ass:2} of Theorem \ref{thm: Affh0 Finite Energy Path_time_constant} as special cases.
Furthermore, as already pointed out in Remark \ref{rem:continuity_time_constant}, we emphasize that the continuity condition for $H$ in Theorem~\ref{prop:Affh0_Finite_Energy_Path} is not actually necessary, but we only consider continuous functions here as this is the condition for $\eich$ in Theorem~\ref{thm: AffWFRDuality_time}.}

\begin{proof}[Proof of Theorem~\ref{thm:general_finite_energy_existence}]
{We start by assuming that the function $\eich=1$ and that $\ef$ is everywhere positive. We then choose some} finite energy path $(\rho,\omega, 0)$ between $\frac{1}{\rho_0(\Omega)}\rho_0$ and $\frac{1}{\rho_1(\Omega)}\rho_1$, which exists since $\Omega$ is a convex compact set and this is balanced, unconstrained transport \cite[Theorem 5.28]{santam}. Since it has finite energy, $\rho$ and $\omega$ are disintegrable in time, $\rho=dt \otimes \rho_t, \omega=dt \otimes \omega_t$. Now define a new path between $\rho_0$ and $\rho_1$ by $\mu=(\bar{\rho},\bar{\omega},\bar{\zeta})=(\ef(t)\rho, \ef(t)\omega, \ef'(t)\rho)$. We show that $\mu$ satisfies the continuity equation as follows: for any $\phi\in C^{1}([0,1]\times \Omega)$,
     \begin{align}
         \int_{0}^{1}\int_{\Omega}\partial_t\phi(t,x) \ef(t)d\rho_t(x) dt &= \int_{0}^{1}\int_{\Omega}(\partial_t(\phi(t,x)\ef(t))-\phi(t,x)\ef'(t))d\rho_t(x) dt \\ 
         &= C-\int_{0}^{1}\int_{\Omega}\nabla\phi(t,x)\ef(t)\cdot d\omega_t(x) dt -\int_{0}^{1}\int_{\Omega}\phi(t,x) \ef'(t)d\rho_t(x) dt \label{eqn:nonzero_case_1}\\ 
         &= C-\int_{0}^{1}\int_{\Omega}\nabla \phi(t,x)\cdot d\bar{\omega}_t(x) dt -\int_{0}^{1}\int_{\Omega}\phi(t,x)d\zeta_t(x) dt ,\label{eqn:nonzero_case_2}
     \end{align}
     where $C=\int_{\Omega}\phi(1,\cdot)d\rho_1-\int_{\Omega}\phi(0,\cdot)d\rho_0$. Equation \eqref{eqn:nonzero_case_1} follows from the fact that $(\rho,\omega, 0)$ satisfies the continuity equation and by substituting $\phi\mapsto \phi(t,x)\ef(t)$. While this function is not of regularity $C^1$ in general, we can approximate it by $C^1$ functions by a density argument. By definition, $\bar{\omega}, \bar{\zeta} \ll \bar{\rho}$ so taking $\lambda=\bar{\rho}$ allows us to calculate the energy. We have
     \begin{align}
         \int_{[0,1]\times \Omega} f_{\delta}\left(\frac{d\mu}{d\lambda}\right)d \lambda & =  \int_{[0,1]\times \Omega}f_\delta\left(\frac{d\bar{\rho}} {d\lambda}, \frac{d\bar{\omega}} {d\lambda}, 0\right)d\lambda + \delta^2\int_{[0,1]\times \Omega}f_\delta\left(\frac{d\bar{\rho}} {d\lambda}, 0, \frac{d\bar{\zeta}} {d\lambda}\right)d\lambda \\ &=  \int_{[0,1]\times \Omega} \ef(t)f_\delta\left(\frac{d\rho} {d\lambda}, \frac{d\omega} {d\lambda}, 0\right)d\lambda + \frac{\delta^2}{2}\int_{0}^{1} \frac{\ef'(t)^2}{\ef(t)}dt \\ &\leq M\int_{[0,1]\times \Omega} f_\delta\left(\frac{d\rho} {d\lambda}, \frac{d\omega} {d\lambda}, 0\right)d\lambda + \frac{\delta^2}{2m}\|\ef'(t)\|_{L^2}^2
     \end{align}
     where $M=\textrm{max}_{t\in[0,1]}\Ft$ and $m=\textrm{min}_{t\in[0,1]}\Ft>0$ which exists because $\Ft$ is continuous on a compact domain $[0,1]$. Here, the continuity of $\ef$ follows from the assumption that $F\in H^1([0,1])$ and the Sobolev embedding theorem.
     The last expression is finite because the first term is the energy of $(\rho,\omega,0)$, and the second term is finite by the $L^2$ assumption on the derivative{, which concludes the proof assuming that $\eich=1$ and that $\ef>0$.}

    Next, we will show the existence of finite energy parts still assuming that $\eich=1$, but assuming that $\ef$ has at least one zero. w.l.o.g. we then assume that $\ef(.5)=0$. Note that we do not assume that this is the only zero of $\ef$. In this case{,} we will simply scale both $\rho_0$ and $\rho_1$ onto the zero measure and concatenate these paths. We only present the construction for $\rho_1${;} the construction for $\rho_0$ is equivalent. For $t\in[\frac12,1]$ we then
     define the path \(\mu=(\rho, \omega, \zeta) = (dt \otimes \rho_t, dt \otimes \omega_t, dt \otimes \zeta_t)\) by 
    \[
        \rho_t = \frac{\Ft}{\ef(1)}\rho_1, \qquad \omega_t = 0, \quad \mbox{and} \quad \zeta_t = \frac{\ef'(t)}{\ef(1)}\rho_1.
    \]
    It is easy to check that this path satisfies the continuity equation and that $\zeta \ll \rho$. Now, take {$\lambda=dt\otimes \rho_1$, so that
    \[
        \frac{d\rho}{d\lambda} = \frac{F(t)}{F(1)}, \qquad \frac{d\omega}{d\lambda} = 0, \quad \mbox{and} \quad \frac{d\zeta}{d\lambda} = \frac{\ef'(t)}{\ef(1)}.
    \]
    Therefore,
    \begin{equation}
        \int_{[\frac12,1]\times \Omega}f_{\delta}\left(\frac{d\mu}{d\lambda}\right)d\lambda = \int_{[\frac12,1]\times \Omega} \frac{\delta^2 \ef'(t)^2}{2F(1)\ef(t)}d\lambda 
        = \int_{\frac12}^{1}\int_{\Omega} \frac{\delta^2 \ef'(t)^2}{2F(1)\ef(t)}d\rho_1 dt 
        = \frac{\delta^2}{2}\int_{\frac12}^{1}\frac{\ef'(t)^2}{\ef(t)}dt,
    \end{equation}
    where the last equality follows from $\ef(1)=\rho_1(\Omega)$.} {The finiteness of the last integral is precisely (the restriction to $[\frac12,1]$ of)  assumption~\eqref{int_assumption}.} 
    
    
    {Thus we have concluded the proof assuming that $\eich=1$. To obtain the result for general $\eich$ we now consider} a finite energy path $(\rho, \omega, \zeta)$ in $\mathcal{CE}_{1,\ef}(\rho_0,\rho_1)$ that connects the measure $\eich(0,\cdot) \rho_0$ to the measure $\eich(1,\cdot) \rho_1$---{this exists by the above arguments}. Then we scale the path by $g=1/\eich$; that is, we consider the path $\tilde{\mu} = (\tilde{\rho},\tilde{\omega},\tilde{\zeta}) = (g\rho, g\omega, g\zeta + (\partial_t g) \rho + (\nabla_x g)\cdot \omega)$ where the source term is modified so that it satisfies the continuity equation. We first note that $g$ is bounded. Indeed, since $\eich$ is continuous and positive-valued, compactness of $[0,1]\times \Omega$ implies that there exists $c > 0$ such that $\Htx \geq c$ for all $(t,x) \in [0,1]\times \Omega$, and we have $|g|= g = 1/\eich\leq 1/c$. Furthermore, the derivatives of $g$ are bounded. Indeed, $\eich$ is assumed to have bounded derivatives, so there exists a constant $M$ such that $|\partial_t g| = |\partial_t 1/\eich| = |\partial_t \eich/\eich^2| \leq M/c^2$ and $\nabla_x \eich$ is also bounded by the same argument. \\
The measure $\tilde{\rho}$ satisfies the desired affine constraint because $d\tilde{\rho}_t(x) = g(t,x)d\rho_t(x)$, so 
\begin{equation}
\int_{\Omega}\Htx d\tilde{\rho}_t(x) =\int_{\Omega}\Htx g(t,x)d\rho_t(x) 
=\int_{\Omega}d\rho_t(x)=\rho_t(\Omega)=\Ft.
\end{equation}
We now show that the path $\tilde{\mu}$ satisfies the continuity equation. For any $\phi\in C^1([0,1]\times \Omega)$, by the Product Rule,
\begin{align}
\int_{[0,1]\times \Omega}(\partial_t \phi )d\tilde{\rho} &= \int_{[0,1]\times \Omega}(\partial_t \phi )gd\rho =\int_{[0,1]\times \Omega}(\partial_t(\phi g)-\phi\partial_t g)d\rho \\
\int_{[0,1]\times \Omega}(\nabla_x \phi )\cdot d\tilde{\omega} &= \int_{[0,1]\times \Omega}(\nabla_x \phi )\cdot gd\omega =\int_{[0,1]\times \Omega}(\nabla_x(\phi g)-\phi\nabla_x g)\cdot d\omega
\end{align}
Moreover, by definition,
\begin{align}
    \int_{[0,1]\times \Omega}\phi d\tilde{\zeta}-\int_{[0,1]\times \Omega}\phi\partial_t g d\rho -\int_{[0,1]\times \Omega}\phi(\nabla_x g)\cdot d\omega = \int_{[0,1]\times \Omega}\phi gd\zeta
\end{align}
Therefore,
\begin{align}
    &\int_{[0,1]\times \Omega}\partial_t \phi d\tilde{\rho} + \int_{[0,1]\times \Omega}(\nabla_x \phi )\cdot d\tilde{\omega}+ \int_{[0,1]\times \Omega}\phi d\tilde{\zeta} \\ 
    &= \int_{[0,1]\times \Omega}\partial_t(\phi g )d\rho + \int_{[0,1]\times \Omega}\nabla_x (\phi g)\cdot d\omega +\left(\int_{[0,1]\times \Omega}\phi d\tilde{\zeta}-\int_{[0,1]\times \Omega}\phi\partial_t g d\rho -\int_{[0,1]\times \Omega}\phi(\nabla_x g)\cdot d\omega\right) \\ 
    &= \int_{[0,1]\times \Omega}\partial_t(\phi g )d\rho + \int_{[0,1]\times \Omega}\nabla_x (\phi g)\cdot d\omega + \int_{[0,1]\times \Omega}\phi g d\zeta \label{eqn: continuity_phig_LHS}\\ 
    &= \int_{\Omega}\phi(1,\cdot)g(1,\cdot)d\rho_1 - \int_{\Omega}\phi(1,\cdot)g(0,\cdot)d\rho_0 \label{eqn: continuity_phig_RHS} = \int_{\Omega}\phi(1,\cdot)d\tilde{\rho}_1 -\int_{\Omega}\phi(0,\cdot)d\tilde{\rho}_0,
\end{align}
where we have used the assumption that $(\rho,\omega, \zeta)$ satisfies the continuity equation. Although $\phi g$ is not $C^1$ in general, we can approximate the function by a $C^1$ function by a density argument.

Next, we show that the path $\tilde{\mu}$ has a finite energy. Take any Borel nonnegative measure $\lambda$ such that  $\rho, \omega, \zeta \ll \lambda$. Then we obtain that
\begin{align}
    \frac{d \tilde{\rho}}{d\lambda} = g\frac{d\rho}{d\lambda}, \ \
    \frac{d \tilde{\omega}}{d\lambda} = g\frac{d\omega}{d\lambda} \ \ \text{and }
    \frac{d \tilde{\zeta}}{d\lambda} = g\frac{d\zeta}{d\lambda} + \partial_t g \frac{d\rho}{d\lambda} + (\nabla_x g)\cdot \frac{d\omega}{d\lambda}
\end{align}
For brevity, let us denote $\tilde{\rho}_\lambda = \frac{d \tilde{\rho}}{d\lambda}, \tilde{\omega}_\lambda = \frac{d \tilde{\omega}}{d\lambda}, \tilde{\zeta}_\lambda = \frac{d \tilde{\zeta}}{d\lambda}$ and similarly for Radon-Nykodim derivatives of $(\rho,\omega,\zeta)$. Observe that the \emph{Wasserstein term} in the WFR objective is bounded as
\begin{equation}\label{eqn:wasserstein_energy}
    \int_{[0,1]\times \Omega} \frac{\|\tilde{\omega}_\lambda \|^2}{\tilde{\rho}_\lambda }d\lambda  = \int_{[0,1]\times \Omega} g\frac{\|\omega_\lambda \|^2}{\rho_\lambda }d\lambda \leq \frac{1}{c} \int_{[0,1]\times \Omega} \frac{\|\omega_\lambda \|^2}{\rho_\lambda }d\lambda,
\end{equation}
where $c>0$ is the lower bound of $h$ introduced earlier. Since the quantity on the right is a constant multiple of the Wasserstein term in the WFR objective evaluated on $(\rho, \omega, \zeta)$, the right hand side of \eqref{eqn:wasserstein_energy} is finite. Now, for the \emph{Fisher-Rao term} in the energy, we have 
\begin{align}&\frac{\tilde{\zeta}_\lambda^2}{\tilde{\rho}_\lambda} =  \frac{g^2\zeta_\lambda^2 + (\partial_t g)^2 \rho_\lambda^2 + (\nabla_x g\cdot \omega_\lambda)^2 +2g\partial_t g \zeta_\lambda \rho_\lambda +2 \partial_t g \rho_\lambda \nabla_x g\cdot \omega_\lambda +2g\zeta_\lambda (\nabla_x g \cdot \omega_\lambda)}{g\rho_\lambda}\end{align}
We now investigate the integral with respect to $\lambda$ of each term of $\frac{\tilde{\zeta}_\lambda^2}{\tilde{\rho}_\lambda}$.
\begin{itemize}[leftmargin=*]
    \item $\frac{g^2 \zeta_\lambda^2}{g\rho_\lambda } = g\frac{\zeta_\lambda^2}{\rho_\lambda }$ has a finite integral because $g$ is bounded and $(\rho,\omega,\zeta)$ has finite energy.
    \item $\frac{(\partial_t g)^2 \rho_\lambda^2}{\rho_\lambda} = \frac{(\partial_t g)^2}{g}\rho_\lambda$ has a finite integral as we integrate a bounded function $\frac{(\partial_t {g})^2}{g}$ by a finite measure $\rho$. 
    \item By Cauchy-Schwartz, we have $\frac{(\nabla_x g\cdot \omega_\lambda)^2}{g\rho_\lambda}\leq \frac{\|\nabla_x g \|^2\|\omega_\lambda\|^2}{g\rho_\lambda} = \frac{g^4\|\nabla_x \eich \|^2\|\omega_\lambda\|^2}{g\rho_\lambda} = g^3\frac{\|\nabla_x \eich \|^2\|\omega_\lambda\|^2}{\rho_\lambda}$ which has a finite integral because of the boundedness and the finite energy assumption. Here, we used $\|\nabla_x g\|= \frac{\|\nabla_x \eich\|}{\eich^2}=g^2\|\nabla_x \eich\|$.

    \item $\frac{2g (\partial_t g) \zeta_\lambda \rho_\lambda}{g\rho_\lambda} = 2(\partial_t g) \zeta_\lambda$ has a finite integral because it is an integral of a bounded function $2\partial_t g$ by a finite measure $\zeta$.

    \item $\frac{2(\partial_t g)\rho_\lambda \nabla_x g\cdot \omega_\lambda}{g\rho_\lambda} = \frac{2(\partial_t g) \nabla_x g\cdot \omega_\lambda}{g} = \left(\frac{2(\partial_t g)\nabla_x g}{g}\right)\cdot \omega_\lambda$ has a finite integral because it is an integral of a bounded function $\frac{2\partial_t g\nabla_x g}{g}$ by a bounded vector measure $\omega${, i.e., $\omega$ is a vector measure with $\|\omega\|(\Omega)<\infty$ \cite[p. 5]{diestel1977vector} where $\|\omega\|$ is the semivariation of $\omega$. This condition is true since it is equivalent to the fact that the range of $\omega$ is bounded \cite[Proposition 11-(b)]{diestel1977vector} and each component of $\omega$ is a Radon measure on a compact set $\Omega$. By the integral inequality $\|\int f d\omega \| \leq \|f\|_{\infty}\|\omega\|(\Omega)$ in \cite[p. 6]{diestel1977vector}, we obtain the finiteness.}

    \item {For the last term, if $\Ft=0$ at some point we can take a finite energy path $(\rho,\omega,\zeta)$ with $\omega=0$ so the last term is zero, and it has a finite integral. If $\Ft>0$ with $\ef'(t)\in L^2$, we can take the path $(\rho,\omega,\zeta)= (\ef(t)\rho^B, \ef(t)\omega^B, \ef'(t)\rho^B)$ where $(\rho^B,\omega^B)$ is a finite energy path in a balanced case as discussed before. }Therefore, 
    \begin{equation}
        \frac{2g\zeta_\lambda (\nabla_x g \cdot \omega_\lambda)}{g\rho_\lambda} = \frac{2g\ef' \rho_\lambda^B (\ef\nabla_x g\cdot \omega_\lambda^B)}{g\ef\rho_\lambda^B} = 2\ef'\nabla_x g \cdot \omega_\lambda^B = 2\ef'\sqrt{\rho_\lambda^B}\cdot \frac{\nabla_x g\cdot \omega_\lambda^B}{\sqrt{\rho_\lambda^B}}
    \end{equation}
    Now, by Cauchy-Schwartz,
    \begin{equation}
        \left\|\ef'(t)\sqrt{\rho_\lambda^B}\cdot \frac{\nabla_x g\cdot \omega_\lambda^B}{\sqrt{\rho_\lambda^B}}\right\|_{L^1(\lambda)} \leq \left\|f'(t)\sqrt{\rho_\lambda^B}\right\|_{L^2(\lambda)}\left\|\frac{\nabla_x g\cdot \omega_\lambda^B}{\sqrt{\rho_\lambda^B}}\right\|_{L^2(\lambda)}
    \end{equation}
    and each term on the right-hand side is finite because
    \begin{align}
        \left\|\ef'\sqrt{\rho_\lambda^B}\right\|^{2}_{L^2(\lambda)} &= \int_{\Omega\times [0,1]}\ef'^2 \rho_\lambda^B d\lambda =\int_{\Omega\times [0,1]}\ef'^2 d\rho^B =\int_{0}^{1}\int_{\Omega}\ef'^2d\rho_t^B dt \\ &= \int_{0}^{1}\ef'^2\rho_t^B(\Omega)dt  = \int_{0}^{1}\ef'^2\cdot 1dt = \|\ef'\|_{L^2}^2<\infty
    \end{align}
    Note that $\rho_t^B(\Omega)=1$ for all $t$ since $\rho^B$ is a path for balanced transport. The other term is finite because
    \begin{align}
    \left\|\frac{\nabla_x g\cdot\omega_\lambda^B}{\sqrt{\rho_\lambda^B}}\right\|^2_{L^2(\lambda)} &= \int_{[0,1]\times \Omega}  \frac{(\nabla_x g \cdot \omega_\lambda^B)^2}{\rho_\lambda^B}d\lambda 
    \end{align}
    which is finite by using Cauchy-Schwartz again and applying the boundedness and the finite energy assumption. Overall, the path has a finite energy. \qedhere
\end{itemize}
\end{proof}
  {
\begin{remark}[A bound for the WFR distance with $\eich=1$ and $\ef>0$.]
  In the case that $\eich=1$ and $\ef>0$ the proof of  Theorem~\ref{thm:general_finite_energy_existence} leads to the upper bound
        \begin{equation}
        \normalfont{\WFRGM}\rhosq \leq M\cdot \mathrm{W}_{2}\left(\frac{\rho_0}{\rho_0(\Omega)},\frac{\rho_1}{\rho_1(\Omega)}\right)^2 + \frac{\delta^2}{2m}\|\ef'(t)\|^2_{L^2},
        \end{equation}
        where $M=\max_{t\in[0,1]}\ef(t), m=\min_{t\in[0,1]}\ef(t)$ and $\mathrm{W}_2$ is the Wasserstein 2-distance (i.e., the square root of \eqref{eqn:kantorovich_ot}).
\end{remark}}

{The results of Theorem~\ref{thm:general_finite_energy_existence} presented several sufficient conditions that guarantee the existence of finite energy paths. In this theorem we assumed that the function $\ef$ is in $H^1([0,1])$ and thus in particular absolutely continuous. We will conclude this section by showing that in the case of $\eich$ being constant absolute continuity for $\ef$ is indeed necessary:}
\begin{remark}
    \label{lem: abs_cont_f}
    Suppose that $\mu=(\rho,\omega ,\zeta)\in \mathcal{CE}_{1,\ef}(\rho_0,\rho_1)$ is a finite energy path. We claim that there exists an absolutely continuous function $u:[0,1]\to \mathbb{R}$ such that $u(t)=\ef(t)$ for almost every $t \in [0,1]$. Indeed, since $\mu$ is a finite energy path, Proposition \ref{prop: disintegration in time} ensures that $\rho$, $\omega$ and $\zeta$ are disintegrable in time, i.e., $\rho=dt \otimes \rho_t,\omega=dt \otimes \omega_t, \zeta = dt \otimes \zeta_t$.  Consider a $C^1$ function which is independent of the space variables, $\phi=\phi(t)\in C^1([0,1])$. Evaluating the continuity equation \eqref{eqn:dist_continuity_equation} with such a $\phi$ gives
    \begin{equation}\label{eq: distributional_derivative}
        \int_{0}^{1} \phi'(t) \rho_t(\Omega)dt + \int_{0}^{1}\phi(t) \zeta_t(\Omega)dt = \phi(1)\rho_1(\Omega) - \phi(0)\rho_0(\Omega)
    \end{equation}
    Define $u:[0,1]\to \mathbb{R}$ by $u(t) = \int_{0}^{t}\zeta_s(\Omega)ds+\rho_0(\Omega)$. Then $u$ is absolutely continuous, and $u'(t)=\zeta_t(\Omega)$ almost everywhere. Applying \eqref{eq: distributional_derivative} to the constant function $\psi(t) = 1$ implies that $\int_{0}^{1}\zeta_t(\Omega)dt = \rho_1(\Omega)-\rho_0(\Omega)$, so $u(1)=\rho_1(\Omega), u(0)=\rho_0(\Omega)$. Integration by parts yields
    \begin{align}
        \int_{0}^{1}\phi(t) \zeta_t(\Omega)dt &= \int_{0}^{1}\phi(t) u'(t)dt = \phi(1)u(1)-\phi(0)u(0)-\int_{0}^{1}\phi'(t)u(t)dt
    \end{align}
    Substituting this into \eqref{eq: distributional_derivative} gives us 
    \begin{align}
        \int_{0}^{1}\phi'(t)\rho_t(\Omega)dt = \int_{0}^{1}\phi'(t) u(t)dt.
    \end{align}
    Since this equation is true for any $\phi(t)\in C^1([0,1])$, we obtain $\rho_t(\Omega)=\ef(t)=u(t)$ almost everywhere.
\end{remark}


\subsection{Optimality certification}
In the following, we examine sufficient conditions for $\mu$ to be an optimal point of the constrained unbalanced optimal transport problem.
\subsubsection{{Optimality certification for Time-dependent constraints}}
In the spirit of Theorem 2.3 of \cite{chizat2018interpolating} in the unconstrained case, we first recall an important result on which we will rely to prove an optimality certification theorem:
\begin{lemma}[Lemma 2.1, \cite{chizat2018interpolating}]
    Denote the objective integral in (\ref{eqn:WFR_distance}) as $D_\delta(\mu)$. The subdifferential of $D_\delta$ at $\mu=(\rho, \omega, \mu)$ such that $D_\delta(\mu) < +\infty$ contains the set
    \begin{equation}
        \partial^c D_\delta (\mu) = \left\{(\alpha,\beta,\gamma)\in C([0,1]\times \Omega;B_{\delta})\middle| \alpha+\frac{1}{2}\left(|\beta|^2+\frac{\gamma^2}{\delta^2}\right)=0-\rho \textrm{ a.e.},\beta\rho=\omega \textrm{ and }\gamma\rho=\delta^2\zeta\right\}
    \end{equation}
\end{lemma}
To prove the main result of this section{,} we will need one more technical lemma. In the following we slightly abuse the notation $\AHF$ and use it for the convex set of all  $\mu = (\rho,\omega,\zeta) \in \Mplus([0,1]\times \Omega)\times \Meas([0,1]\times \Omega)^n \times \Meas([0,1]\times \Omega)$ such that $\rho \in \AHF$. 
\begin{lemma}
    Let $\eich = (H_1, \cdots H_d) \in C([0,1]\times \Omega)^d$ and $F:[0,1]\times \Omega \to \mathbb{R}^d$. The normal cone {$N(\mu)$} of {$\CEHF$}, i.e., the subdifferential of {$\iota_{\CEHF}$} at $\mu=(\rho,\omega,\zeta) \in {\CEHF,}$ {contains the set} 
    \begin{equation}
    {N_\rho(\mu)}=\left\{\left(\sum_{i=1}^{d}H_i(t,x)g_i(t),0,0\right)\in C([0,1]\times \Omega; \mathbb{R}\times\mathbb{R}^n\times \mathbb{R})|g_i\in C([0,1]),i=1,\cdots ,d\right\}
    \end{equation} 
\end{lemma}
{
\begin{proof}
    Let $\mu = (\rho, \omega, \zeta) \in \CEHF$. The normal cone of $\CEHF$ at $\mu$ is the set of all $(\alpha, \beta, \gamma) \in C([0,1]\times \Omega; \mathbb{R}\times\mathbb{R}^n\times \mathbb{R})$ such that
    \begin{equation}
        \langle (\alpha, \beta, \gamma), \nu - \mu \rangle \leq 0 \quad \forall \nu \in \CEHF.
    \end{equation}
    Equivalently, for any $\nu = (\bar{\rho}, \bar{\omega}, \bar{\zeta}) \in \CEHF$, we should have
    \begin{equation}
        \int_{[0, 1]\times \Omega}\alpha d(\bar{\rho}-\rho) + \int_{[0, 1]\times \Omega}\beta \cdot d(\bar{\omega} - \omega) + \int_{[0, 1]\times \Omega}\gamma d(\bar{\zeta} - \zeta) \leq 0.
    \end{equation}
    By substituting $\alpha = \sum_{i=1}^{d}H_i(t,x)g_i(t)$, $\beta = 0$ and $\gamma = 0$ for some $g_i \in C([0,1])$, we have
    \begin{align}
        \int_{[0, 1]\times \Omega}\sum_{i=1}^{d}H_i(t,x)g_i(t) d(\bar{\rho}-\rho) = \sum_{i=1}^{d}\int_{0}^{1}g_i(t)\left(\int_{\Omega}H_i(t,x)d\bar{\rho}_t - \int_{\Omega}H_i(t,x)d\rho_t\right)dt = 0
    \end{align}
    Therefore, $(\sum_{i=1}^{d}H_i(t,x)g_i(t),0,0) \in N(\mu)$.
\end{proof}
}

We are now able to show our main result of this subsection:
\begin{theorem}[Optimality certificate]
\label{thm:opt_certificate}
    Let $\rho_0,\rho_1\in \mathcal{M}_+(\Omega)$. For $\mu=(\rho,\omega, \zeta)\in \CEHF$, if there exists $\phi \in C^1([0,1]\times \Omega)$ such that 
    \begin{equation}
        (\partial_t\phi, \nabla\phi, \phi) \in \partial^c D_\delta(\mu)+N(\mu) 
    \end{equation}
    then $\mu$ is a 
    minimizer for the infimum in $(\ref{eqn: constWFR_affine})$. We refer to such $\phi$ as an optimality certificate for $\mu$.
\end{theorem}
\begin{proof}
    We note that $\partial^cD_\delta(\mu)+N(\mu)\subset \partial D_\delta(\mu)+\partial\iota_{\CEHF}(\mu)\subset \partial(D_\delta+\iota_{\CEHF})(\mu)$, so if $(\partial_t\phi, \nabla\phi, \phi) \in \partial^c D_\delta(\mu)+N(\mu)$, the triple belongs to the subgradient of $D_\delta(\mu)+\iota_{\AHF}$, i.e., for any $\mu'\in {\CEHF}$,
    \begin{align}
        D_\delta(\mu')- D_\delta(\mu)\geq \langle(\partial_t\phi, \nabla\phi, \phi), \mu'-\mu\rangle =0 
    \end{align}
    where the right-hand side equals zero by definition of the distributional continuity equation. Therefore, we have that $D_\delta(\mu')\geq D_\delta(\mu)$ for any $\mu'\in \CEHF$, which concludes the proof. 
\end{proof}
\begin{example}[Scaling with prescribed total mass]
As an application of the above optimality certification, we will find a minimizer of a pure scaling with a prescribed total mass constraint, i.e., let $\ef:[0,1]\to (0,\infty)$ be a $C^2$ function and consider the affine constraint $\AFF(1,\ef)$. Furthermore let $\rho_0\in \mathcal{M}^+(\Omega)$ such that $\rho_0(\Omega)=\ef(0)$ and let $\rho_1=\frac{\ef(1)}{\ef(0)}\rho_0$ be a pure scaling of $\rho_0$. Then 
\begin{equation}
    \rho = dt \otimes \frac{\ef(t)}{\ef(0)}\rho_0
\end{equation}
defines a minimizer for $(\ref{eqn: constWFR_affine})$ with $\eich=1$, and 
\begin{equation}
    \WFRGM(\rho_0,\rho_1) = \frac{\delta}{\sqrt{2}}\left(\int_{0}^{1} \frac{\ef'(t)^2}{\ef(t)}dt\right)^{1/2}.
\end{equation}
To see this we  define the function $\phi \in C^1([0,1]\times \Omega)$ by 
    \begin{equation}
    \phi(t,x) = \delta^2 \frac{\ef'(t)}{\ef(t)}.
    \end{equation}
    We obtain 
    \begin{equation}
        (\partial_t \phi, \nabla\phi, \phi) =\left( \delta^2 \frac{\ef''(t)\ef(t)- (\ef'(t))^2}{\ef(t)^2}, 0, \delta^2 \frac{\ef'(t)}{\ef(t)}\right)
    \end{equation}
    We will see that this triple decomposes to the sum of elements of $\partial^c D_\delta(\mu)$ and $N(\mu)$ where $\mu =(\rho, 0, dt\otimes \frac{f'(t)}{f(0)}\rho_0)$. Write down $\partial^c D_\delta(\mu)$ in this special case to see that
    \begin{equation}
        \partial^c D_\delta (\mu) = \left\{(\alpha,\beta,\gamma)\middle| \alpha+\frac{1}{2}\left(|\beta|^2+\frac{\gamma^2}{\delta^2}\right)=0-\rho \textrm{ a.e.},\beta=0-\rho\textrm{ a.e.} \textrm{ and }\gamma=\delta^2\frac{\ef'(t)}{\ef(t)}-\rho\textrm{ a.e.}\right\}
    \end{equation}
    Therefore, we observe that 
    \begin{equation}
        \left(-\frac{\delta^2}{2}\frac{\ef'(t)^2}{\ef(t)^2}, 0, \delta^2\frac{\ef'(t)}{\ef(t)}\right) \in \partial^c D_\delta(\mu)
    \end{equation}
    We will now consider $N(\mu)$. Since $\eich=1$, we have $N(\mu)=\{(g,0,0)|g\in C([0,1])\}$, so  
    \begin{equation}
        \left( \delta^2\frac{\ef''(t)\ef(t)- (1/2)(\ef'(t))^2}{\ef(t)^2}, 0, 0 \right)\in N(\mu)
    \end{equation}
    Finally, we have 
    \begin{align}
     &\left(-\frac{\delta^2}{2}\frac{\ef'(t)^2}{\ef(t)^2}, 0, \delta^2\frac{\ef'(t)}{\ef(t)}\right) + \left( \delta^2\frac{\ef''(t)\ef(t)- (1/2)(\ef'(t))^2}{\ef(t)^2}, 0, 0 \right) \\&= \left( \delta^2 \frac{\ef''(t)\ef(t)- (\ef'(t))^2}{\ef(t)^2}, 0, \delta^2 \frac{\ef'(t)}{\ef(t)}\right) = (\partial_t \phi, \nabla\phi, \phi) \in \partial^cD_\delta(\mu) +N(\mu)
     \end{align}
     Thus, $\phi$ is an optimal certificate, and $\mu$ is a minimizer of the problem. We now calculate the optimal energy. By definition, $\rho$ and the source term $dt\otimes \frac{\ef'(t)}{\ef(0)}\rho_0$ are absolutely continuous with respect to $\lambda =dt\otimes \rho_0$, so
    \begin{equation}
        \WFRGM\rhosq = \delta^2\int_{0}^{1}\int_{\Omega} \frac{(\ef'(t)/\ef(0))^2}{2(\ef(t)/\ef(0))}d\rho_0dt = \frac{\delta^2}{2} \int_{0}^{1} \frac{\ef'(t)^2}{\ef(t)}dt. 
    \end{equation}
\end{example}
{
\subsubsection{Optimality certification for time-independent constraints}
\label{sec:suff-cond-optimal}
In the situation of time-independent constraint functions $H$ and $F$, and assuming sufficient regularity on $H$, one can derive a slightly different form for the sufficient optimality certificate by expressing the constraints on $\omega$ and $\zeta$ instead of $\rho$, as in~Lemma \ref{lem:affine_constraint_rho_zeta}. This will allow us to directly relate it to the geodesic equations formally derived in Section \ref{sec:constrained_WFR_distance}. We first prove the following:
\begin{lemma}
    Let $H \in C^1(\Omega)$ and $C \in \mathbb{R}^d$. The normal cone $N(\mu)$ of $\CEHC$, i.e., the subdifferential of $\iota_{\CEHC}$ at $\mu=(\rho,\omega,\zeta) \in \CEHC$, contains the set
    \begin{equation}
        N_{\omega,\zeta}(\mu) = \left\{\left(0, \sum_{i=1}^{d}\nabla H_i(x) g_i(t), \sum_{i=1}^{d}H_i(x) g_i(t)\right)\middle| g_i\in C([0,1])\right\}.
    \end{equation}
\end{lemma}
\begin{proof}
    Let $\mu = (\rho, \omega, \zeta) \in \CEHC$. The normal cone of $\CEHC$ at $\mu$ is the set of all $(\alpha, \beta, \gamma) \in C([0,1]\times \Omega; \mathbb{R}\times\mathbb{R}^n\times \mathbb{R})$ such that
    \begin{equation}
        \langle (\alpha, \beta, \gamma), \nu - \mu \rangle \leq 0 \quad \forall \nu \in \CEHC.
    \end{equation}
    Equivalently, for any $\nu = (\bar{\rho}, \bar{\omega}, \bar{\zeta}) \in \CEHC$, we should have
    \begin{equation}
        \int_{[0, 1]\times \Omega}\alpha d(\bar{\rho}-\rho) + \int_{[0, 1]\times \Omega}\beta \cdot d(\bar{\omega} - \omega) + \int_{[0, 1]\times \Omega}\gamma d(\bar{\zeta} - \zeta) \leq 0.
    \end{equation}
    By substituting $\alpha = 0$, $\beta = \sum_{i=1}^{d}\nabla H_i(x)g_i(t)$ and $\gamma = \sum_{i=1}^{d}H_i(x)g_i(t)$ for some $g_i \in C([0,1])$, we have
    \begin{align}
        \int_{[0, 1]\times \Omega}\sum_{i=1}^{d}\nabla H_i(x)g_i(t) d(\bar{\omega}-\omega) + \int_{[0, 1]\times \Omega}\sum_{i=1}^{d}H_i(x)g_i(t) d(\bar{\zeta}-\zeta) = 0
    \end{align}
    by Lemma \ref{lem:affine_constraint_rho_zeta}. Therefore, $(0, \sum_{i=1}^{d}\nabla H_i(x) g_i(t), \sum_{i=1}^{d}H_i(x) g_i(t)) \in N(\mu)$.
\end{proof}
We will now show that the geodesic equation is a sufficient condition for optimality.
\begin{theorem}
\label{thm:opt_cond_stat}
    Let $\rho_0, \rho_1 \in \MplushC, H\in C^1(\Omega)$ and $C \in \mathbb{R}^d$. For $\mu=(\rho,\omega, \zeta)\in \CEHC$, if there exist $\Phi \in C^1([0,1]\times \Omega), g_i\in C([0,1]),i=1,\cdots, d$ such that
    \begin{align}
        &\partial_t \Phi + \frac{1}{2}\|\nabla \Phi(t, \cdot) - \nabla \bar{\Phi}(t, \cdot)\|^2 + \frac{1}{2\delta^2}(\Phi(t, \cdot) - \bar{\Phi}(t, \cdot))^2 \leq 0, \ \text{with equality } \rho \text{-a.e.} \\
        &\omega = (\nabla \Phi -  \nabla \bar{\Phi})\rho, \\  
        &\zeta = \frac{1}{\delta^2}(\Phi - \bar{\Phi})\rho
    \end{align}
    with $\bar{\Phi}(t,x) = \sum_{i=1}^{d}H_i(x)g_i(t)$, then $\mu$ is a minimizer of $(\ref{eqn: constWFR_affine})$. 
\end{theorem}
\begin{proof}
We indeed have $(\partial_t \Phi, \nabla \Phi - \nabla \bar{\Phi}, \Phi - \bar{\Phi}) \in \partial^c D_\delta(\mu)$ from the first condition, while on the other hand one has $(0, \sum_{i=1}^{d} \nabla H_i(x) g_i(t), \sum_{i=1}^{d}H_i(x) g_i(t)) \in N_{\omega, \zeta} \subset N(\mu)$. It follows that: 
\begin{equation*}
    (\partial_t \Phi, \nabla \Phi, \Phi) = (\partial_t \Phi, \nabla \Phi - \nabla \bar{\Phi}, \Phi - \bar{\Phi}) + (0, \sum_{i=1}^{d} \nabla H_i(x) g_i(t), \sum_{i=1}^{d}H_i(x) g_i(t)) \in \partial^c D_\delta(\mu) + N(\mu)
\end{equation*}
and thus Theorem \ref{thm:opt_certificate} implies that $\mu$ is a minimizer.
\end{proof}
\begin{remark}
It is worth noting that under the conditions of Theorem \ref{thm:opt_cond_stat} on $\omega$ and $\zeta$, and using the fact that $\mu \in \AFF(H, C)$, it follows necessarily that for all $i=1,\ldots,d$ and a.e. $t \in [0,1]$,
\begin{align*}
    \int_{\Omega} \nabla H_i \cdot (\nabla \Phi - \nabla \bar{\Phi}) d\rho_t +\frac{1}{\delta^2} \int_{\Omega} H_i(\Phi - \bar{\Phi}) d\rho_t = 0,
\end{align*}
which can be rewritten, using the notations of Section \ref{sec:constrained_WFR_distance}, as:
\begin{align*}
    \sum_{j=1}^{d} \langle H_i, H_j \rangle_{W^{1,2}_{\rho_t}} \, g_j(t) = \langle H_i, \Phi(t,\cdot) \rangle_{W^{1,2}_{\rho_t}}.
\end{align*}
This is exactly the same linear system on $(g_i(t))$ as \eqref{eq:expr_bar_phi} and thus $\bar{\Phi}(t,\cdot)$ is again the orthogonal projection of $\Phi(t,\cdot)$ onto the space spanned by the $H_i$'s with respect to $\langle\cdot,\cdot\rangle_{W^{1,2}_{\rho_t}}$.    
As a consequence, the above sufficient condition of Theorem \ref{thm:opt_cond_stat} becomes identical to the necessary condition \eqref{eq:opt_cond2} in the case where $(\rho,\omega,\zeta)$ are all densities since the continuity equation then writes as:
    \begin{equation*}
        \partial_t \rho + \nabla\cdot (\nabla \Phi - \nabla \bar{\Phi})\rho = \frac{1}{\delta^2}(\Phi - \bar{\Phi})\rho
    \end{equation*}
    for all $(t,x)$ in $[0,1] \times \Omega$.
\end{remark}
}

\section{Constrained unbalanced optimal transport on Riemannian manifolds}\label{sec:manifolds}
In this section{,} we will briefly outline how one can generalize the developed theory to the manifold domain case, albeit under certain restrictive assumptions. Let us denote by $M$ a Riemannian manifold of dimension $m$. There are several specific issues with dealing with measures on $M$. {The} first one is the definition of the transport measure $\omega$. Following the setting of \cite{gallouet2021regularity}, we define a momentum field on $M$ as follows:
\begin{definition}[Momentum field] A \textbf{momentum field} on a Riemannian manifold $M$ is an element of the dual to $\Gamma_0([0,1] \times M)$, the space of continuous functions $v:[0,1]\times M\to TM$ such that $\pi \circ v(t,x)=x$ for any $t\in [0,1]$ and $x\in M$ where $\pi: TM\to M$ is the usual projection. We will denote the space of momentum fields by $\mathcal{V}([0,1]\times M)$.
\end{definition}
This definition generalizes the concept of $\mathbb{R}^n$-valued vector measures by taking the Riesz-Markov-Kakutani theorem \cite[Theorem 7.17]{folland2013real} as motivation.
Just as an Euclidean momentum field $\omega \in \mathcal{M}([0,1]\times \Omega)^n$ is the dual of $\mathbb{R}^n$-valued continuous maps by the aforementioned theorem, we use the dual of continuous vector fields for the momentum field. This allows us to extend the continuity equation to the Riemannian case:
\begin{definition}[Riemannian continuity equation]
    Let $\rho_0,\rho_1\in \mathcal{M}^+(M)$. We say that a triple $(\rho,\omega,\zeta)\in \mathcal{M}^+([0,1]\times M)\times \mathcal{V}([0,1]\times M)\times \mathcal{M}([0,1]\times M)$ satisfies the continuity equation $\partial_{t} \rho +\nabla\cdot \omega = \zeta$ with boundary conditions  $\rho_0$ and $\rho_1$ in the distributional sense if, for any $\phi \in C^1([0,1]\times M)$,  
    \begin{equation}
        \int_{[0,1]\times M}\partial_{t}\phi d\rho +\omega(\nabla_x \phi) + \int_{[0,1]\times M}\phi d\zeta = \int_{M}\phi(1,\cdot)d\rho_1-\int_{M}\phi(0,\cdot)d\rho_0 
    \end{equation}
\end{definition}
We also define the set of triples $(\rho,\omega,\zeta)$ satisfying the continuity equation above as $\mathcal{CE}(\rho_0,\rho_1)$. It can be shown, in the same way as in the Euclidean situation (Proposition \ref{prop: disintegration in time}), that if $\rho$ satisfies the continuity equation, then $\rho$ can be disintegrated in time with respect to the Lebesgue measure{,} i.e.{,} $\rho = dt \otimes \rho_t$. We can thus extend the previous notation and again write $\CEHF$ for the set of all triples $(\rho,\omega,\zeta) \in \mathcal{CE}(\rho_0,\rho_1)$ that further satisfy the constraint $\int_M \eich d\rho_t = \ef(t)$ for a.e. $t \in [0,1]$.

Given two measures $\rho_0,\rho_1 \in \mathcal{M}^+(M)$, we may then adapt the previous definition to the manifold case and introduce the following constrained unbalanced OT problem:
\begin{equation}
\label{eq:constrained_UOT_M}
\WFRAF(\rho_0,\rho_1)^2 =\inf\left\{ \int_{[0,1] \times M} f_{\delta,x}\left(\frac{d \mu}{d \lambda}(t,x)\right) d \lambda(t,x) \ \middle| \ \mu=(\rho,\omega,\zeta)\in \CEHF(\rho_0,\rho_1)\right\}
\end{equation}
where $\lambda$ is any nonnegative measure such that $\rho,\omega,\zeta \ll \lambda$ and $f_{\delta,x}:\mathbb{R} \times T_xM^* \times \mathbb{R} \rightarrow [0, +\infty]$ is given by:
\begin{equation}
\label{eq:def_fdel_M}
   f_{\delta,x}(a,b,c) = \begin{cases}
        \frac{g_x^{-1}(b, b)+\delta^2 c^2}{2a}, & a>0 \\ 0, & (a,b,c)=(0,0,0) \\ +\infty, & \textrm{otherwise}
    \end{cases}
\end{equation}
where $g_x$ denotes the Riemannian metric tensor at $x \in M$ while its inverse $g_x^{-1}$ is the cometric. 


As one can see, although the formulation remains quite similar to the Euclidean case, the cost function integrand's dependency on the point $x$ itself in particular requires adapting the convex duality argument used to prove the existence of minimizing paths. First, let us state and prove the following lemma:
\begin{lemma}
\label{lemma:conjugate_manifold_case}
Let $x \in M$ and denote $B_{\delta,x}=\left\{(\alpha,\beta,\gamma) \in \mathbb{R} \times T_x M \times \mathbb{R} \ | \ \alpha + \frac{1}{2} \left(g_x(\beta,\beta) + \frac{\gamma^2}{\delta^2}\right) \leq 0\right\}$. The Fenchel conjugate of the convex indicator $\iota_{B_{\delta,x}}$ is the function $f_{\delta,x}$ defined in \eqref{eq:def_fdel_M}. 
\end{lemma}
\begin{proof}
    By definition of the Fenchel conjugate, we have: 
    \begin{equation*}
       f_{\delta,x}(a,b,c) = \sup_{(\alpha,\beta,\gamma)} a \alpha + \langle b, \beta \rangle + c \gamma - \iota_{B_{\delta,x}}(\alpha,\beta,\gamma) = \sup_{(\alpha,\beta,\gamma) \in B_{\delta,x}} a \alpha + \langle b, \beta \rangle + c \gamma
    \end{equation*}
    First, note that if $(a,b,c)=(0,0,0)$ then obviously $f_{\delta,x}(a,b,c)=0$. If $a<0$, taking $(\alpha,\beta,\gamma) = (ta,0,0)$, and letting $t \rightarrow +\infty$ shows that $f_{\delta,x}(a,b,c)=+\infty$. In addition, when $a=0$ but $(b,c) \neq (0,0)$, it suffices to set $(\beta,\gamma) = t (\bar{b},c)$, where $\bar{b}\in T_x M$ is such that $\langle b, v \rangle = g_x(\bar{b},v)$ for any $v \in T_x M$, and $\alpha = -\frac{1}{2}\left(g_x(t \bar{b},t \bar{b}) + \frac{(tc)^2}{\delta^2}\right)$. Then letting $t \rightarrow + \infty $ leads again to $f_{\delta,x}(a,b,c)=+\infty$. Let us thus consider the last remaining case in which $a>0$. For any $(\alpha,\beta,\gamma) \in B_{\delta,x}$, one has $\alpha \leq -\frac{1}{2} \left(g_x(\beta,\beta) + \frac{\gamma^2}{\delta^2}\right)$ so that:
    \begin{equation*}
       a \alpha + \langle b, \beta \rangle + c \gamma \leq -\frac{a}{2} \left(g_x(\beta,\beta) + \frac{\gamma^2}{\delta^2}\right) + \langle b, \beta \rangle + c \gamma
    \end{equation*}
    The maximum of the right hand side with respect to $\gamma$ is achieved at $\hat{\gamma} = \delta^2 \frac{c}{a}$ for which we find $-\frac{a \hat{\gamma}^2}{2\delta^2} + c \hat{\gamma} = \frac{\delta^2 c^2}{2a}$. Similarly, the maximum value with respect to $\beta$ is achieved for $\hat{\beta} = \frac{1}{a}\bar{b}$ in which case: 
    \begin{equation*}
        -\frac{a}{2} g_x(\hat{\beta},\hat{\beta}) + \langle b, \beta \rangle = \frac{g_x(\bar{b},\bar{b})}{2a} = \frac{g_x^{-1}(b,b)}{2a}
    \end{equation*}
    Combined with the above, we therefore obtain $f_{\delta,x}(a,b,c)=\frac{g_x^{-1}(b,b) + \delta^2 c^2}{2a}$.
\end{proof}

Note that $f_{\delta,x}$ is also a 1-homogeneous function. We can now state a generalized version of Theorem \ref{thm: AffWFRDuality_time}. In the following, the manifold $M$ is said to be parallelizable if it is possible to find a set of $m$ continuous tangent vector fields on $M$ written $\{v_1,\ldots,v_m\}$ such that for all $x \in M$, $\{v_1(x),\ldots,v_m(x)\}$ is a basis of $T_x M$.
\begin{theorem}[Existence of minimizers on parallelizable manifolds]
    \label{thm:AffWFRDuality_time_M} 
    Assume that $M$ is a parallelizable manifold and let $\rho_0,\rho_1\in\Mplus(M)$ and $\eich:[0,1]\times M \to \mathbb{R}^d$ be a continuous function, $\ef=(\ef_1,\ldots,\ef_d):[0,1]\to \mathbb{R}^d$ be a function with integrable component functions $\ef_i$. If there exists a finite energy path in $\CEHF$, then there exists a minimizing path to problem \eqref{eq:constrained_UOT_M}.
\end{theorem}
\begin{proof}
    The proof in part follows a similar path as that of Theorem \ref{thm: AffWFRDuality_time} above, so we will mainly highlight the steps which differ. We shall apply again the Fenchel-Rockafellar theorem to the primal problem: $\inf_{\tau=(\phi,\psi)} {\mathcal F}(A \tau) + {\mathcal G}(\tau)$ in which $\tau\in X=C^{1}([0,1] \times M) \times C([0,1])^d$, $Y=C([0,1] \times M) \times \Gamma^0([0,1] \times M) \times C([0,1] \times M)$ and $A$ is the bounded linear operator $A:X \rightarrow Y$ defined by $A(\phi,\psi) = (\partial_t \phi + \sum_{i=1}^{d} \eich_i \psi_i,\nabla \phi, \phi)$. The functional ${\mathcal G}:X \rightarrow \mathbb{R} \cup \{+\infty\}$ is defined in the same way as in the proof of Theorem \ref{thm: AffWFRDuality_time} while ${\mathcal F}:Y \rightarrow \mathbb{R} \cup \{+\infty\}$ is given by:
    $$
    {\mathcal F}(\alpha,\beta,\gamma) = \int_{[0,1] \times M} \iota_{B_{\delta,x}}(\alpha(t,x),\beta(t,x),\gamma(t,x)) dt \, d\text{vol}_M(x)
    $$
    with $\text{vol}_M$ being the Riemannian volume measure on $M$. Note that both ${\mathcal F}$ and ${\mathcal G}$ are convex functionals and that, by definition, $Y^*$ can be identified with the product of measure spaces $\mathcal{M}([0,1]\times M)\times \mathcal{V}([0,1]\times M)\times \mathcal{M}([0,1]\times M)$. Fenchel-Rockafellar theorem once again applies so that: 
    $$
    \inf_{\tau \in X} \{{\mathcal F}(A \tau) + {\mathcal G}(\tau) \} = \sup_{\mu \in Y^*} \{-{\mathcal F}^*(\mu) - {\mathcal G}^*(-A^* \mu)\}
    $$
    The term involving the convex conjugate of ${\mathcal G}$ can be obtained in the same way as previously and one gets that ${\mathcal G}^*(-A^*\mu)$ is again the convex indicator of the constraint set $\CEHF$. An extra difficulty compared to the Euclidean case is to recover the expression of ${\mathcal F}^*(\mu)$. We here rely on the parallalizability assumption on $M$ and the existence of continuous tangent vector fields $\{v_1,\ldots,v_m\}$ introduced above. This allows us to identify $\Gamma_0([0,1] \times M)$ with $C([0,1] \times M)^m$. Indeed, for any $\beta \in \Gamma_0([0,1] \times M)$, we may decompose $\beta(t,x)$ for $(t,x) \in [0,1] \times M$ in the basis $\{v_1(x),\ldots,v_m(x)\}$ as $\beta(t,x) = \sum_{i=1}^{m} \beta_i(t,x) v_i(x)$ and thus identify $\beta$ with $(\beta_1,\ldots,\beta_m) \in C([0,1] \times M)^m$. Under this convention, $(\alpha,\beta,\gamma)$ maps $(t,x) \in [0,1]\times M$ to the fixed Euclidean space $\mathbb{R}^{m+2}$ and \cite[Theorem 5]{rockafellar_integrals_1971} still applies to the integral functional $F$ in this case as well. Using Lemma \ref{lemma:conjugate_manifold_case} above together with the fact that $f_{\delta,x}$ is 1-homogeneous (which implies that it is equal to its recession function), we then get:
    \begin{equation*}
        {\mathcal F}^*(\mu) = \int_{[0,1] \times M} f_{\delta,x}\left(\frac{d \mu}{d \lambda}(t,x)\right) d \lambda(t,x).
    \end{equation*}
    From this, it follows once again that the dual problem $\sup_{\mu \in Y^*} \{-{\mathcal F}^*(\mu) - {\mathcal G}^*(-A^* \mu)\} = -\inf_{\mu \in Y^*} \{{\mathcal F}^*(\mu) + {\mathcal G}^*(-A^* \mu)\}$ is equivalent to the infimum of the original problem \eqref{eq:constrained_UOT_M} which allows to conclude the proof.
\end{proof}
It is worth pointing out that the assumption made on $M$ in Theorem \ref{thm:AffWFRDuality_time_M} includes manifolds such as {Lie groups (for instance,} $S^1$, $S^3$ and higher-dimensional tori{)} as well as manifolds that are the image of a single parametrization function on an open domain of an Euclidean space. However, it does not cover several important cases, {including the 2-sphere, which is not parallelizable}. We suspect however that this could be extended to general compact manifolds based on a careful refinement of the results from \cite{rockafellar_integrals_1971} that we used in the proof.  

Regarding the existence of a finite energy path between two given measures $\rho_0$ and $\rho_1$ in $\mathcal{M}^+(M)$, adapting the proofs of Theorems \ref{thm: Affh0 Finite Energy Path_time_constant}, \ref{thm:general_finite_energy_existence} and \ref{prop:Affh0_Finite_Energy_Path} to the manifold case is rather straightforward as most of the arguments do not rely on the particular Euclidean domain structure of $\Omega$. The statements of these theorems extend mutatis mutandis by replacing $\Omega$ with a compact Riemannian manifold $M$ and replacing the convexity assumption on $\Omega$ (when needed) with $M$ being geodesically complete.  

\begin{example}[Area measures of convex domains]
\label{ex:area_measures}
 An example of measure space which naturally fits the manifold UOT framework of this section is the case of the area measures of convex domains in $\mathbb{R}^d$ mentioned earlier in the introduction, which was also previously considered by some of the authors in \cite{charon_length_2021} and \cite{bauer2022square}. {In the following, we will briefly describe the connection with the present work but shall leave aside the technical details related to the definition and properties of area measures:}  to any oriented convex domain $Q$ in $\mathbb{R}^d$, one can associate its area measure $\rho_Q$ which is the positive Radon measure on the unit sphere $S^{d-1}$ (viewed as a submanifold of $\mathbb{R}^d$) such that for any Borel set $B$, $\rho_Q(B)$ measures the $(d-1)$-area of the portion of the boundary $\partial Q$ where the unit outward normal $\vec{n}(x) \in B$. A central result in convex geometry is that the area measure characterizes a convex setup to translation \cite{fenchel1938}. In fact, $Q \mapsto \rho_Q$ induces a bijective correspondence between convex sets modulo translations and the space of all measures $\Meas^+(S^{d-1})$ that satisfy the ``closure'' constraint $\int_{S^{d-1}} x d\rho(x) = 0 \in \mathbb{R}^d$. By defining $H:x\in S^{d-1} \mapsto x \in \mathbb{R}^d$, the latter space is precisely $\Meas^+_{H,0}(S^{d-1})$ in the setting of this paper. Thus, by pulling back the metric $\WFRAFho$ on $\Meas^+_{H,0}(S^{d-1})$, one can in turn view the space of all convex shapes as a geodesic space with some potentially interesting characteristics of the associated metric and the geodesics, which were hinted in \cite{charon_length_2021} in the case $d=2$. While our present results mainly address the question of {the} existence of geodesics, we leave a more in depth study of the properties of such metrics to future work.
\end{example}

\bibliographystyle{alpha}  
\bibliography{citation} 
\end{document}